\documentclass[a4paper,12pt]{amsart}
\usepackage[utf8]{inputenc}
\usepackage[T1]{fontenc}
\usepackage[UKenglish]{babel}
\usepackage[margin=20mm]{geometry}
\usepackage{color}
\usepackage{bbm,bm}

\newcommand{\red}{\color{red}}

\usepackage{graphicx}

\usepackage{amsmath,amssymb,amsfonts,amsthm}
\usepackage{mathrsfs,eucal,dsfont}
\usepackage{verbatim,enumitem}

\usepackage{hyperref,url}
\usepackage{amssymb, amsmath, amsthm, amscd}

\usepackage{eucal,mathrsfs,dsfont}
\usepackage{color}

\renewcommand{\leq}{\le}
\renewcommand{\geq}{\ge}

\newcommand{\I}{\mathds 1}

\def\d{{\rm d}}
\def\<{\langle}
\def\>{\rangle}

\newtheorem{theorem}{Theorem}[section]
\newtheorem{lemma}[theorem]{Lemma}
\newtheorem{proposition}[theorem]{Proposition}

\numberwithin{equation}{section}

\theoremstyle{definition}
\newtheorem{definition}[theorem]{Definition}

\newtheorem{remark}[theorem]{Remark}

\begin{document}
\allowdisplaybreaks
\title[Geodesic of the critical long-range percolation metric]
{\bfseries  Uniqueness and dimension for the geodesic of the critical long-range percolation metric}

\author{Jian Ding \qquad  Zherui Fan  \qquad  Lu-Jing Huang}

\thanks{\emph{J. Ding:}
School of Mathematical Sciences, Peking University, Beijing, China.
  \texttt{dingjian@math.pku.edu.cn}}
\thanks{\emph{Z. Fan:}
School of Mathematical Sciences, Peking University, Beijing, China.
  \texttt{zheruifan@pku.edu.cn}}

\thanks{\emph{L.-J. Huang:}
School of Mathematics and Statistics \& Key Laboratory of Analytical Mathematics and Applications
(Ministry of Education), Fujian Normal University, Fuzhou, China.
  \texttt{huanglj@fjnu.edu.cn}}


\date{}
\maketitle

\begin{abstract}
By recent works of B\"aumler \cite{ Ba23} and of the authors of this paper \cite{DFH23+}, the (limiting) random metric for the critical long-range percolation was constructed. In this paper, we prove the uniqueness of the geodesic between two fixed points, for which an important ingredient of independent interest is the continuity of the metric distribution.  In addition, we establish the Hausdorff dimension of the geodesics. 


\medskip


\end{abstract}
\allowdisplaybreaks


\section{Introduction}\label{intro}

Consider the critical long-range percolation (LRP) on $\mathds{Z}^d$, where edges $\langle \bm i,\bm j\rangle $ with $\|\bm i-\bm j\|_\infty=1$ (i.e.\ $\bm i$ and $\bm j$ are nearest neighbors) occur independently with probability 1, while edges $\langle \bm i,\bm j\rangle $ with $\|\bm i-\bm j\|_\infty>1$ (which we will refer to as long edges) occur independently with probability given by
\begin{equation*}\label{LRP}
1-\exp\left\{-\beta\int_{V_{1}(\bm i)}\int_{V(\bm j)}\frac{1}{|\bm u-\bm v|^2}\d \bm u\d \bm v\right\}.
\end{equation*}
Here, $V_1(\bm i)$ represents the cube in $\mathds{R}^d$ centered at $\bm i$ with side length 1, $\|\cdot\|_\infty$ denotes the $\ell^\infty$-norm, and $|\cdot|$ represents the $\ell^2$-norm.
This model is commonly referred to as the critical LRP model or the $\beta$-LRP model.
In this paper, we are interested in its metric properties.

Regarding the metric properties of the critical LRP model, \cite{CGS02} first established that the typical diameter of a box grows at a polynomial rate with respect to its side length, with the growth rate being strictly less than 1.
Additionally, \cite{Ba23} showed (see also \cite{DS13} for the one-dimensional case) that the typical distance between two points $\bm 0$ and $n\bm 1$ (here $\bm 1=(1,1,\cdots, 1)\in \mathds{R}^d$) grows as $n^\theta$ for some $\theta\in(0,1)$. More precisely, denote by $d(\cdot,\cdot)$ the graph distance (also known as the chemical distance) in this LRP model. \cite{Ba23} proved that
\begin{equation}\label{def-theta}
n^\theta\asymp_P d(\bm 0,n\bm 1)\asymp_P \mathds{E}\left[d(\bm 0,n\bm 1)\right],
\end{equation}
where $\theta\in(0,1)$ is a constant  depending only on $d$ and $\beta$, and $A(n)\asymp_P B(n)$ means that for each $\varepsilon>0$, there exist constants $0<c<C<\infty$ such that $\mathds{P}[cB(n)\leq A(n)\leq CB(n)]\geq 1-\varepsilon$ for all $n\in \mathds{N}$.
It is worth emphasizing that this work implies the existence of the subsequential scaling limit of the graph distance. Recently, in \cite{DFH23+} we proved the uniqueness of such subsequential limits, thereby completing the construction of the limiting metric for the critical LRP.

\begin{theorem}[{\cite[Theorem 1.1]{DFH23+}}]
For any $n\in \mathds{N}$, let $a_n$ be the median of $d(\bm 0,n\bm 1)$ and $D_n(\cdot,\cdot)=a_n^{-1}d(\lfloor n\cdot\rfloor,\lfloor n\cdot\rfloor)$.
Then there exists a unique random metric $D$ on $\mathds{R}^d$ such that $D_n$ converges to $D$ in law with respect to the topology of local uniform convergence on $\mathds{R}^{2d}$.
\end{theorem}

Furthermore, we have established a set of axioms for the random metric $D$, which will be referred to as the $\beta$-LRP metric.
These axioms emcompass properties such as locality, translation and scale covariance, tightness and others,
see \cite[Sections 1.3 and 1.4]{DFH23+}. In particular, we have provided an upper bound for the Hausdorff dimension of the $\beta$-LRP metric.
This upper bound has recently been shown by \cite{Ba23a} to be the lower bound, thereby providing a definitive value for the Hausdorff dimension of the $\beta$-LRP metric.
Additionally, \cite{Ba23a} has discussed further properties related to the balls in the graph metric.
In this paper, we will further explore some other important properties of the $\beta$-LRP metric $D$.

Our first main result establishes the continuity of the distribution of the $\beta$-LRP metric $D$, as detailed in the following theorem. Let $\mathds{S}=\{\bm x\in \mathds{R}^d:\ |\bm x|=1\}$.

\begin{theorem}[Continuity of the distribution of $\beta$-LRP metric]\label{thm-continuity}
For all $d\geq 1$ and $\beta>0$, we have
\begin{equation*}
\lim_{\varepsilon\to 0^+}\mathds{P}\left[D({\bm 0},\bm x)\in (a-\varepsilon,a+\varepsilon)\right]=0
\end{equation*}
for all $a\geq 0$ and $\bm x\in \mathds{S}$.
\end{theorem}

\begin{remark}
Due to the translation and scale covariance of the LRP metric established in \cite[Axiom IV in Section 1.3]{DFH23+}, the statement in Theorem \ref{thm-continuity} holds for all distinct $\bm x,\bm y\in \mathds{R}^d$.
\end{remark}

Next, we focus on the geodesics of the $\beta$-LRP metric $D$.
For convenience, we will refer to a $D$-geodesic as the geodesic associated with the metric $D$.

\begin{theorem}[Uniqueness of $D$-geodesics] \label{thm-unique}
For all $d\geq 1$,  $\beta>0$ and for any fixed pair of distinct points ${\bm x},{\bm y}\in \mathds{R}^d$, almost surely, there exists a unique $D$-geodesic from ${\bm x}$ to ${\bm y}$.
\end{theorem}

We now consider the Hausdorff dimension of $D$-geodesics.
To this end, we begin with some general notations.
Let $(Y, \widehat{D})$ be a metric space.
For any $\Delta > 0$, the $\Delta$-Hausdorff content of $(Y, \widehat{D})$ is defined by
\begin{equation}\label{hausdorff}
	C_\Delta(Y,\widehat{D})=\inf\left\{\sum_{j=1}^\infty r_j^\Delta:\text{ there is a cover of $Y$ by $\widehat{D}$-balls of radii $\{r_j\}_{j\geq 1}$}\right\}.
\end{equation}
The Hausdorff dimension of $(Y, \widehat{D})$ is then defined as
$$
\text{dim}_{\rm H}(Y;\widehat{D})=\inf\left\{\Delta > 0 : C_\Delta(Y, \widehat{D}) = 0\right\}.
$$
For simplicity, if $Y$ is a subset of $\mathds{R}^d$ and $\widehat{D}$ is the Euclidean distance, we denote the Hausdorff dimension of $(Y,\widehat{D})$ by $\mathrm{dim}_{\rm{H}}(Y)$.

In the following, we will take $Y={\rm Range}(P)$ for some $D$-geodesic $P$. 
The following theorem provides the value of $\mathrm{dim}_{\rm H}(\mathrm{Range}(P))$ for each $D$-geodesic $P$.

\begin{theorem}\label{mr-dim}
	For all $d\geq 1$ and $\beta>0$, almost surely, we have $\mathrm{dim}_{\mathrm{H}}(\mathrm{Range}(P))=\theta$ for all $D$-geodesics $P$, where $\theta \in (0,1)$ {\rm(}depending only on $d$ and $\beta${\rm)} is the exponent defined in \eqref{def-theta}.
\end{theorem}

The organization of this article is as follows.
In Section \ref{sperner}, we present a generalized version of a Sperner's theorem. Using this as an important ingredient, we prove Theorem \ref{thm-continuity} in Section \ref{sect-T1}, from which we derive Theorem \ref{thm-unique} in Section \ref{sect-T2}. Finally, in Section \ref{sect-dim}, we provide the proof of Theorem \ref{mr-dim}.

\bigskip

\textbf{Notational conventions.}
We denote $\mathds{N}=\{1,2,\cdots\}$. For $a<b$, we define $[a,b]_\mathds{Z}=[a,b]\cap \mathds{Z}$ and $\mathds{Q}_+$ as the set of positive rational numbers.
Write $\textbf{1}=(1,1,\cdots, 1)\in \mathds{R}^d$. Throughout the paper, we use the boldface font (e.g.\ $\bm x,\bm y,\bm u, \bm v, \bm i, \bm j$) to represent vectors on $\mathds{R}^d$. For any $r>0$ and $\bm x\in \mathds{R}^d$,  denote by $B_r(\bm x)$ the ball in $\mathds{R}^d$ centered at $\bm x$ with radius $r$, and by $V_r(\bm x)$ the cube in $\mathds{R}^d$ centered at $\bm x$ with side length $r$.

For any metric $\widetilde{D}$ and subset $U$, denote $\text{diam}(U; \widetilde{D})$ as the diameter of $U$ with respect to $\widetilde{D}$.

\section{General Sperner family and associated estimates}\label{sperner}

Before presenting our version of Sperner's theorem, we recall the definition of the classical Sperner family. 

\begin{definition}\label{def-sperner}
For $n\in \mathds{N}$, let $\mathcal{A}_n$ be a collection of subsets of $[1,n]_{\mathds{Z}}$. 
We say $\mathcal{A}_n$ is a \textit{Sperner family} 
if for each $A\in \mathcal{A}_n$ either of the following holds:
\begin{itemize}

\item[(1)] There exists $B_A\subset [1,n]_{\mathds{Z}} \backslash A$ with $|B_A|\geq \frac{1}{2} n$ such that for all $A' \supset A$ and $A'\cap B_A\neq \emptyset$, we have $A'\notin \mathcal{A}_n$;
\medskip

\item[(2)] There exists $B'_A\subset A$ with $|B'_A|\geq \frac{1}{2} n$ such that for all $A' \subset A$ and $ B'_A\backslash A'\neq \emptyset$, we have $A'\notin \mathcal{A}_n$.

\end{itemize}
\end{definition}


For $n\in \mathds{N}$, we now let $\xi_1,\cdots, \xi_n$ be a sequence of independent Bernoulli random variables with parameter $p$.
A realization of these random variables can be represented by the set $A\subset [1,n]_{\mathds{Z}}$, which records the indices where the random variables take on the value 1.
Based on this representation, we define $\mathcal{A}_n$ as a Sperner event if it forms a Sperner family.

\begin{proposition}\label{prop-sperner}
There exists a constant $C_1<\infty$ depending only on $p$ such that for any Sperner event $\mathcal{A}_n$, 
we have $\mathds{P}[\mathcal{A}_n]\leq \frac{C_1}{\sqrt{n}}$.
\end{proposition}

In order to prove Proposition \ref{prop-sperner}, we need a more extensive version of Sperner's theorem.
We note that a generalization of Sperner's theorem is presented in \cite[Theorem 4.2]{DS20}, which built upon a modification of Lubell's proof \cite{L66} for the Lubell-Yamamoto-Meshalkin inequality.
We will carry out some (obvious) generalization over \cite{DS20} to cover all $p\in (0, 1)$.
To this end, we assume that $\mathcal{A}_n$ is a Sperner family. 
In this context,  we say that $A\in \mathcal{A}_n$ is upward-unstable (resp.\ downward-unstable) if Definition \ref{def-sperner} (1) (resp.\ (2)) holds for $A$.
Let $\mathcal{A}_{n,\rm up}$ (resp.\ $\mathcal{A}_{n,\rm down}$) be the collection of all upward-unstable (resp.\ downward-unstable) sets in $\mathcal{A}_n$.

Let us denote
$$
a_k=|\{A\in \mathcal{A}_n:\ |A|=k\}|\quad \text{for all }0\leq k\leq n.
$$
We then have the following estimate for $\{a_k\}$.

\begin{lemma}[Generalization of Sperner's theorem]\label{spernerlem}
Let $\mathcal{A}_n$ be a Sperner family. 
Then we have $\sum_{k=0}^n \frac{a_k}{\tbinom{n}{k}}\leq 4$.
\end{lemma}

\begin{proof}
	For all $0\leq k\leq n$, denote
	\begin{equation*}
		a_{k,\mathrm{up}}=|\{A\in \mathcal{A}_{n,\mathrm{up}}:\ |A|=k\}|,\quad a_{k,\mathrm{down}}=|\{A\in \mathcal{A}_{n,\mathrm{down}}:\ |A|=k\}|.
	\end{equation*}
	From \cite[Proof of Theorem 4.2]{DS20}, we have
	\begin{equation*}
		\sum_{k=0}^n \frac{a_{k,\mathrm{up}}}{\tbinom{n}{k}}\leq 2\quad \text{and}\quad
		\sum_{k=0}^n \frac{a_{k,\mathrm{down}}}{\tbinom{n}{k}}\leq 2.
	\end{equation*}
	Summing up the two inequality yields the lemma.
\end{proof}

With this lemma at hand, we can present the

\begin{proof}[Proof of Proposition \ref{prop-sperner}]
Let $m\in [1,n]_{\mathds{Z}}$ be the number such that
$$
\tbinom{n}{m} p^m(1-p)^{n-m}=\max_{k}\tbinom{n}{k}p^k(1-p)^{n-k}.
$$
Indeed, it is clear that $m\in \{\lfloor pn\rfloor, \lceil pn\rceil\}$. Combining this with Lemma \ref{spernerlem}, we get that
\begin{equation*}
\begin{aligned}
\mathds{P}[\mathcal{A}_n]&=\sum_{k=1}^n a_k p^k (1-p)^{n-k}
=\tbinom{n}{m} p^m(1-p)^{n-m}\sum_{k=1}^n\frac{a_k p^k (1-p)^{n-k}}{\tbinom{n}{m} p^m(1-p)^{n-m}}\\
&\leq \tbinom{n}{m} p^m(1-p)^{n-m}\sum_{k=1}^n\frac{a_k p^k (1-p)^{n-k}}{\tbinom{n}{k} p^k(1-p)^{n-k}}\leq \frac{\widetilde{C}}{\sqrt{n}}
\end{aligned}
\end{equation*}
for some constant $\widetilde{C}<\infty$ depending only on $p$.
\end{proof}

\section{Proof of Theorem \ref{thm-continuity}}\label{sect-T1}

The objective of this section is to establish Theorem \ref{thm-continuity}. Since the proof of uniqueness in \cite{DFH23+} employed the framework for the Liouville quantum gravity (LQG) metric \cite{GM21, DG23}, we offer a remark here for comparison: the continuity of the distribution for the LQG metric can be proved easily since one can perturb the Gaussian free field by adding a smooth function without drastically changing its probability density. So roughly speaking, it is not surprising and easy to derive a continuous distribution on the distance when the disorder has a continuous distribution to begin with. However, in the case of LRP, on the contrary, the disorder is given by a Poisson point process which by nature is ``discrete''; this is why in our case the continuity result is not at all obvious and requires significant work (including employing the Sperner's theorem). Before carrying out our proof, we introduce some notations which will be used repeatedly.
Recall that $D(\cdot,\cdot; U)$ is the LRP metric restricted to $U$ for all $U\subset \mathds{R}^d$.
For any $r>s>0$, denote $\mathbb{A}_{r,s}= B_r({\bm 0})\setminus B_s({\bm 0})$.
From the tightness of the LRP metric (see \cite[Theorem 1.4]{DFH23+}), it is clear that there exist constants $c_{*,1}, c_{*,2}>0$ (both depending only on $d$ and $\beta$) such that
\begin{equation}\label{di}
\begin{aligned}
&\mathds{P}\left[D(B_{1/2}({\bm 0}), \mathbb{A}_{1,7/8};B_1({\bm 0}))\geq c_{*,1}\right]\\
&\geq \mathds{P}\left[B_{1/2}({\bm 0})\nsim B_{7/8}({\bm 0})^c\right]\mathds{P}\left[D(B_{1/2}({\bm 0}), B_{7/8}({\bm 0})^c)\geq c_{*,1}\ | B_{1/2}({\bm 0})\nsim B_{7/8}({\bm 0})^c\right]\geq c_{*,2}.
\end{aligned}
\end{equation}

In the rest of this section, we fix sufficiently small $\varepsilon \in (0,1)$.
Recall that $\theta\in (0,1)$ is the exponent defined in \eqref{def-theta}.
Denote
\begin{equation}\label{def-KN}
K=\sqrt{\log(1/ \varepsilon)} \quad \text{and}\quad N=\left( 4^{-(1+1/\theta)}(c_{*,1}/\varepsilon)^{1/\theta}\right)^{1/K}.
\end{equation}
In addition, let $\{M_i\}_{i\in [1,K]_{\mathds{Z}}}$ denote a sequence satisfying
\begin{equation}\label{def-M}
M_1=(4\varepsilon/c_{*,1})^{1/\theta}\quad \text{and}\quad M_i=NM_{i-1}\quad \text{for }2\leq i\leq K.
\end{equation}
By some simple calculations and the fact that $\theta\in (0,1)$ (depending only on $d$ and $\beta$), one has that for sufficiently small $\varepsilon \in (0,1)$ (depending only on $d$ and $\beta$),
\begin{equation}\label{condit-Mi}
N\geq 2\quad \text{and}\quad \sum_{i=1}^KM_i\leq 1.
\end{equation}
Define $r_0=0$ and $r_i=\sum_{k=1}^i M_k$ for $1\leq i\leq K$. We also let $\delta>0$ denote a constant satisfying
\begin{equation}\label{delta0}
4\delta^{\theta/2}\leq c_{*,1}.
\end{equation}


\subsection{Sparsity of long edges}
We start by defining the following events on the occurrence of certain long edges (see Figure \ref{fig-AE} for an illustration).
\begin{definition}\label{def-A}
For each $i\in [1,K]_{\mathds{Z}}$,
\begin{itemize}
\item[(1)] let $A_i$ denote the event that there exists at least one long edge connecting $B_{r_{i-1}}({\bm 0})$ and $B_{r_{i-1}+ M_i/8}({\bm 0})^c$;

\item[(2)] let $E_{i,1}$ denote the event that there exists at least one long edge connecting $B_{r_{i}-M_i/8}({\bm 0})$ and $B_{r_i}({\bm 0})^c$. Here we use a subscript 1 in $E_{i, 1}$ since later we will introduce some other events with subscripts 2, 3 and 4.
\end{itemize}
\end{definition}

\begin{figure}[htbp]
\centering
\includegraphics[scale=0.7]{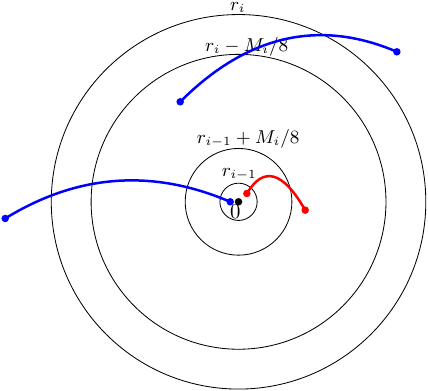}
\caption{The illustration for the events $A_i$ and $E_{i,1}$. The event $A_i$ indicates the presence of red long edges in the picture, while the event $E_{i,1}$ indicates the presence of blue edges in the picture.}
\label{fig-AE}
\end{figure}

Denote $A=\cup_{i=1}^KA_i$. Then we have the following estimate for $\mathds{P}[A]$.

\begin{lemma}\label{lem-probA}
For all $d\geq 1$ and all $\beta>0$, there exists a constant $c_1=c_1(d,\beta)>0$ {\rm(}depending only on $d$ and $\beta${\rm)} such that $\mathds{P}[A_i]\leq c_1N^{-d}$ for all $i\in [1,K]_{\mathds{Z}}$.
Therefore, we have $\mathds{P}[A^c]\geq 1-c_1KN^{-d}$.
\end{lemma}

\begin{proof}
For each $i\in [1,K]_{\mathds{Z}}$, by the definition of $A_i$ in Definition \ref{def-A} (1) and \eqref{condit-Mi}, we get that
\begin{align*}
\mathds{P}[A_i]&=1-\exp\left\{-\int_{B_{r_{i-1}}({\bm 0})}\int_{B_{r_{i-1}+ M_i/8}({\bm 0})^c}\frac{\beta}{|{\bm u}-{\bm v}|^{2d}}\d {\bm u}\d {\bm v}\right\}\\
&\leq 1-\exp\left\{-\int_{B_{r_{i-1}}({\bm 0})}\int_{B_{ M_i/8}({\bm u})^c}\frac{\beta}{|{\bm u}-{\bm v}|^{2d}}\d {\bm u}\d {\bm v}\right\}\\
&=1-\exp\left\{-\beta \widetilde{c}_1(d)r_{i-1}^d \int_{M_i/8}^{+\infty}\frac{1}{s^{d+1}}\d s\right\}\\
&\leq \beta \widetilde{c}_2(d)\left(\frac{r_{i-1}}{ M_i}\right)^d\leq \widetilde{c}_3(d,\beta) N^{-d}
\end{align*}
for some constants $\widetilde{c}_1(d),\ \widetilde{c}_2(d)>0$ (depending only on $d$) and $\widetilde{c}_3(d,\beta)>0$ (depending only on $d$ and $\beta$).
This implies that
$$
\mathds{P}[A^c]\geq 1-\sum_{i=1}^K\mathds{P}[A_i]\geq 1-\widetilde{c}_3(d,\beta)KN^{-d}.
$$
Hence, the proof is complete.
\end{proof}

In the following, we aim to provide a large deviation estimate for the number of occurrences of event $E_{i,1}$, $i\in [1,K]_{\mathds{Z}}$. To do this, we define a sequence of Bernoulli random variables $\{\xi_i\}_{1\leq i\leq K}$ as
\begin{equation*}
\xi_i=
\begin{cases}
1,\quad & E_{i,1}^c\ \text{occurs},\\
0,\quad & \text{otherwise}.
\end{cases}
\end{equation*}

\begin{proposition}\label{numb-Di}
For all $d\geq 1$ and all $\beta>0$, there exists a constant $\kappa=\kappa(d,\beta)\in (0,1)$ {\rm(}depending only on $d$ and $\beta${\rm)} such that
$$
\mathds{P}\left[\sum_{i=1}^K \xi_i\leq (1-\kappa)K/2\right]\leq \exp\left\{-(1-\kappa)^2K/2\right\}.
$$
We will refer to the event $\left\{\sum_{i=1}^K \xi_i\leq (1-\kappa)K/2\right\}$ as $E_1$.
\end{proposition}

To prove Proposition \ref{numb-Di}, we need to do some preparations.

\begin{lemma}\label{prob-bad}
For all $d\geq 1$ and all $\beta>0$, there exists a constant $\kappa\in (0,1)$ {\rm(}depending only on $d$ and $\beta${\rm)} such that the following holds. For any $k\in [1,K]_{\mathds{Z}}$ and any $S=\{i_1,\cdots, i_k\}\subset [1,K]_{\mathds{Z}}$ in the ascending order, let
$$
w_k=\mathds{P}[\xi_{i_1}=0,\cdots, \xi_{i_k}=0].
$$
Then we have $w_k\leq \kappa^k$.
\end{lemma}

The proof of Lemma \ref{prob-bad} is inspired by \cite[Section 2]{DFH25}, which employs the Firework process and view coverage of long edges over annuli as the propagation in the Firework process.
To this end, we fix  $k\in [1,K]_{\mathds{Z}}$ and $S=\{i_1,\cdots, i_k\}\subset [1,K]_{\mathds{Z}}$ being a set in the ascending order.
For $l\in [0,k-1]_{\mathds{Z}}$, let $L_l$ represent the maximum number of $s\geq 1$ such that there exists at least one long edge within
$$
\mathcal{E}_{\mathbb{A}_{r_{i_{l+1}}- M_{i_{l+1}}/8,r_{i_l}-M_{i_l}/8}\times B_{r_{i_{l+s}}}({\bm 0})^c}\neq \emptyset,
$$
where $i_0= r_{i_0}=M_{i_0}:=0$ and $\mathbb{A}_{r_{i_{1}}-M_{i_{1}}/8,r_{i_0}- M_{i_0}/8}:=B_{7M_{i_1}/8}({\bm 0})$ (see Figure \ref{fig-L} for an illustration). It is worth emphasizing that $L_1,\cdots, L_{k-1}$ are independent from the independence of edges in the LRP model.

\begin{figure}[htbp]
\centering
\includegraphics[scale=0.4]{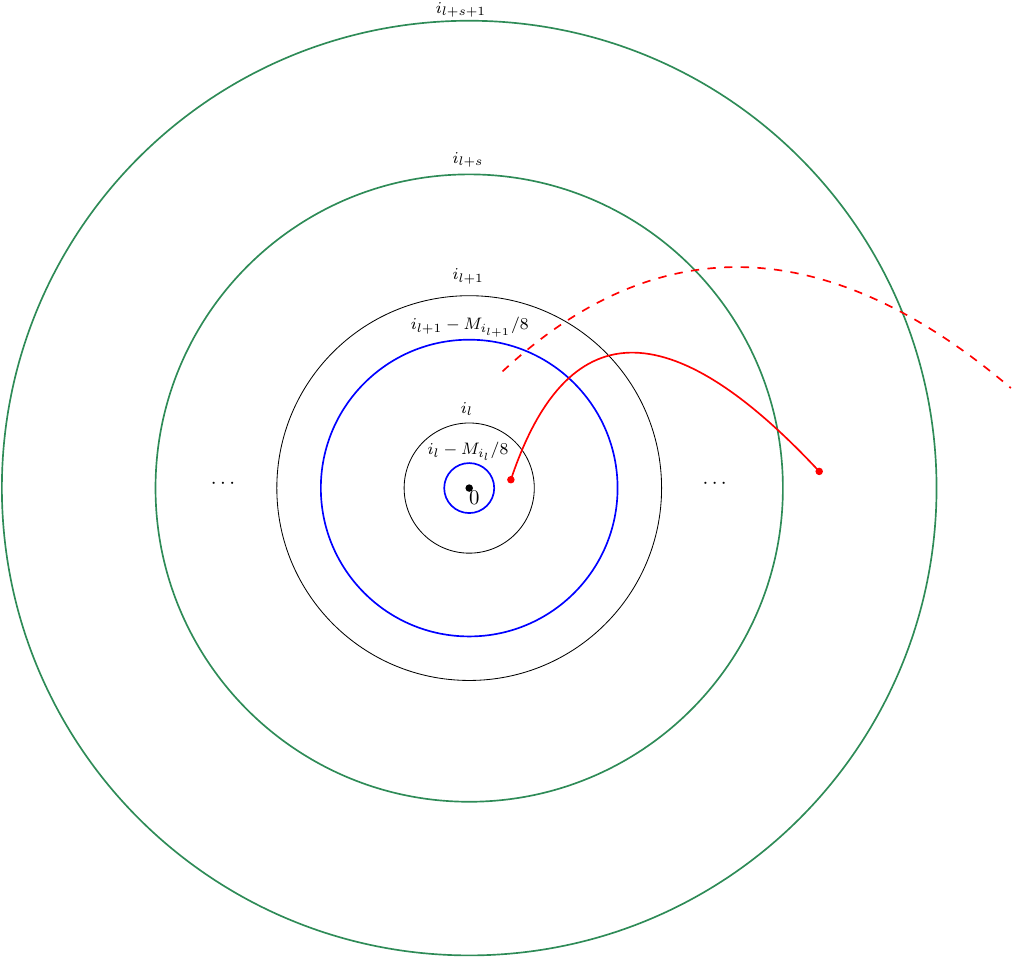}
\caption{The illustration for the definition of $L_{l}$. The event $\{L_l=s\}$ is equivalent to the event that there exists at least one long edge in $\mathcal{E}_{\mathbb{A}_{r_{i_{l+1}}- M_{i_{l+1}}/8,r_{i_l}-M_{i_l}/8}\times B_{r_{i_{l+s}}}({\bm 0})^c}$ (the red curve), while no long edge exists in $\mathcal{E}_{\mathbb{A}_{r_{i_{l+1}}- M_{i_{l+1}}/8,r_{i_l}-M_{i_l}/8}\times B_{r_{i_{l+s+1}}}({\bm 0})^c}$ (the red dashed curve).}
\label{fig-L}
\end{figure}

For the distribution of $L_l$, we have the following property.

\begin{lemma}\label{tail-Ll}
For all $d\geq 1$ and all $\beta>0$, there exists a constant $c_2=c_2(d,\beta)>0$ {\rm(}depending only on $d$ and $\beta${\rm)} such that for each $k\in [1,K]_{\mathds{Z}}$ and each $S=\{i_1,\cdots, i_k\}\subset [1,K]_{\mathds{Z}}$ in the ascending order, we have
$$
\mathds{P}[L_l\geq s]\leq 1-\exp\left\{-\frac{c_2}{1/8+2^{s-1}}\right\}
$$
for all $l\in [0,k-1]_{\mathds{Z}}$ and $s\in [1,k-l]_{\mathds{Z}}$.
\end{lemma}

\begin{proof}
We write $\mathrm{dist}(\cdot, \cdot)$ for the Euclidean distance between two sets. It follows from the definition of $L_l$ and the fact that 
$$
\text{dist}(\mathbb{A}_{r_{i_{l+1}}- M_{i_{l+1}}/8,r_{i_l}-M_{i_l}/8}, B_{r_{i_{l+s}}}({\bm 0})^c)\geq r_{i_{l+s}}-r_{i_{l+1}}+ M_{i_{l+1}}/8
$$
that for each $l\in [0,k-1]_{\mathds{Z}}$ and each $s\in [1,k-l]_{\mathds{Z}}$, there exist constants $\widetilde{c}_1,\cdots,\widetilde{c}_4>0$ (all depending only on $d$ and $\beta$) such that
\begin{align*}
\mathds{P}[L_l\geq s]&
=1-\exp\left\{-\int_{\mathbb{A}_{r_{i_{l+1}}- M_{i_{l+1}}/8,r_{i_l}-M_{i_l}/8}}\int_{B_{r_{i_{l+s}}}({\bm 0})^c}\frac{\beta}{|{\bm u}-{\bm v}|^{2d}}\d {\bm u}\d {\bm v}\right\}\\
&\leq 1-\exp\left\{-\widetilde{c}_1\left((r_{i_{l+1}}- M_{i_{l+1}}/8)^d-(r_{i_l}-M_{i_l}/8)^d\right)\int_{r_{i_{l+s}}-r_{i_{l+1}}+ M_{i_{l+1}}/8 }^{+\infty}\frac{1}{t^{d+1}}\d t\right\}\\
&\leq 1-\exp\left\{-\widetilde{c}_2(r_{i_{l+1}}-M_{i_{l+1}}/8)^d(r_{i_{l+s}}-r_{i_{l+1}}+M_{i_{l+1}}/8)^{-d}\right\}\\
&\leq 1-\exp\left\{-\widetilde{c}_3\frac{2-1/8}{1/8+N^{s-1}}\right\}\leq 1-\exp\left\{-\frac{\widetilde{c}_4}{1/8+2^{s-1}}\right\},
\end{align*}
where the third inequality is from $i_{l+s}-i_{l+1}\geq s-1$ and $M_{i+1}/M_i=N$ (see \eqref{def-M}), and the last inequality is from the fact that $N\geq 2$ in \eqref{condit-Mi}.
\end{proof}

Let $W_0=\{0\}$ and for $m\geq 1$, we inductively define
\begin{equation}\label{def-Wm}
W_m=\{i_s\in S:\ \text{there exists }i_l\in W_{m-1}\ \text{such that }s\leq l+L_l\}\setminus W_{m-1}.
\end{equation}
The above definition in fact corresponding to a ``spreading'' procedure of the edge set for the model, where in the $m$-th step we explore long edges in
\begin{equation}\label{edgeset-Wm-1}
\bigcup_{i_l\in W_{m-1}} \mathcal{E}_{\mathbb{A}_{r_{i_{l+1}}- M_{i_{l+1}}/8,r_{i_l}-M_{i_l}/8}\times B_{r_{i_{s}}}({\bm 0})^c}
\end{equation}
to determine if the element $i_s$ is in $W_m$. That is, if the edge set in \eqref{edgeset-Wm-1} is non-empty, then $i_s\in W_m$.
We see that $i_s\in W_m$ (namely, the annulus $\mathbb{A}_{r_{i_s},r_{i_s}-M_{i_s}/8}$ is newly covered at the $m$-th step) if it was not covered by the edge set
\begin{equation*}
\bigcup_{i_l\in W_{m-2}} \mathcal{E}_{\mathbb{A}_{r_{i_{l+1}}-M_{i_{l+1}}/8,r_{i_l}-M_{i_l}/8}\times B_{r_{i_{s}}}({\bm 0})^c},
\end{equation*}
but covered by the edge set in \eqref{edgeset-Wm-1}. The ``spreading procedure'' will stop upon $W_m=\emptyset$ and from \eqref{def-Wm} we can see that $W_{m'}=\emptyset$ for all $m'>m$ if $W_m=\emptyset$. Moreover, $\cup_{m\geq 0} W_m$ represents the set of subscripts for annuli (``spreaders'') at the end of the above spreading procedure. Let
\begin{equation}\label{def-Mk}
M_k=\min\left\{l\in \mathds{N}:\ i_l\in \cup_{m\geq 0}W_m\right\}
\end{equation}
be the subscript of the last pair of annuli that are covered in this spreading procedure. Combining this with the definition of $\xi_i$, we can see that
\begin{equation}\label{xi-M}
\mathds{P}[\xi_{i_1}=0,\cdots, \xi_{i_k}=0]\leq \mathds{P}[M_k\geq k].
\end{equation}
We will use some estimates for the Firework process to bound the right-hand side of \eqref{xi-M} from above, which in turn provides an upper bound on the left-hand side of \eqref{xi-M}.

\begin{lemma}\label{lem-tail-Mk}
For all $d\geq 1$ and all $\beta>0$, there exists $\kappa\in (0,1)$ {\rm(}depending only on $d$ and $\beta${\rm)} such that for all $k\geq 1$,
$$
\mathds{P}[M_k\geq k]\leq \kappa^k.
$$
\end{lemma}
\begin{proof}
For $l\in [0,k-1]_{\mathds{Z}}$ and $s\in [0,k-l]_{\mathds{Z}}$, denote
$$
\alpha_l(s)=\mathds{P}[L_l\leq s]=1-\mathds{P}[L_l\geq s+1].
$$
Then from Lemma \ref{tail-Ll}, we can see that for each $s\in [1,k]_{\mathds{Z}}$,
\begin{equation}\label{alpha-tilde}
\min_{l\leq k-s}\alpha_l(s)\geq \exp\left\{-\frac{c_2}{1/8+2^{s}}\right\}=:\widetilde{\alpha}(s),
\end{equation}
where $c_2>0$ is the constant (depending only on $d$ and $\beta$) defined in Lemma \ref{tail-Ll}. Moreover, we have
\begin{equation}\label{log-alphatilde}
\log\left(\prod_{s\geq 0} \widetilde{\alpha}(s)\right)=-c_2\sum_{s\geq 0}\frac{1}{1/8+2^s}>-\infty.
\end{equation}

For convenience, we extend the definition of $\widetilde{\alpha}(s)$ in \eqref{alpha-tilde} to $\mathds{R}$. Specifically, we define
\begin{equation*}
\widetilde{\alpha}(s)=
\begin{cases}
0,\quad & s<0,\\
\exp\left\{-\frac{c_2}{1/8+2^{s}}\right\},\quad &s\geq 0.
\end{cases}
\end{equation*}
Then it is easy to check that $\widetilde{\alpha}(s)$ is now a distribution function.
We consider a sequence of i.i.d.\ random variables $\widetilde{L}_0,\cdots, \widetilde{L}_{k-1}$ with the distribution $\widetilde{\alpha}(s)$, and define $\widetilde{W}_m$ and $\widetilde{M}_k$ according to \eqref{def-Wm} and \eqref{def-Mk} by replacing $L_l$ with $\widetilde{L}_l$, respectively. Then from \eqref{alpha-tilde} and the independence of $L_0,\cdots, L_{k-1}$, we can see that
\begin{equation}\label{Mk-Mktilde}
\mathds{P}[M_k\geq s]\leq \mathds{P}\left[\widetilde{M}_k\geq s\right]\quad \text{for all }s\in [0,k]_{\mathds{Z}}.
\end{equation}
Additionally, it follows from  \eqref{alpha-tilde} and \eqref{log-alphatilde} that $\widetilde{\alpha}(s)$ increases exponentially to 1 as $s\to \infty$ and $\prod_{s\geq 0}\widetilde{\alpha}(s)>0$, which implies that conditions stated in \cite[Proposition 1]{GGJR14} are satisfied. Consequently, by applying \cite[Proposition 1]{GGJR14} to $\widetilde{\alpha}(s)$, we get that there exist $\kappa_1\in (0,1)$ and $R>0$ (both depending only on $d$ and $\beta$) such that for all $k\geq R$,
\begin{equation}\label{tail-Mk}
\mathds{P}\left[\widetilde{M}_k\geq k\right]\leq \kappa_1^k.
\end{equation}
Moreover, we choose $\kappa_2\in (0,1)$ (depending only on $d$ and $\beta$) such that $\kappa_2^R\geq 1-\exp\{-8c_2/9\}$.
By combining this with Lemma \ref{tail-Ll}, we see that for all $k<R$,
\begin{equation}\label{tail-Mk2}
\mathds{P}\left[\widetilde{M}_k\geq k\right]\leq \mathds{P}[L_{k-1}\geq 1]\leq 1-\exp\{-8c_2/9\}\leq \kappa_2^k.
\end{equation}
Now, let us denote $\kappa=\max\{\kappa_1,\kappa_2\}$. Then \eqref{tail-Mk} and \eqref{tail-Mk2} yield that $\mathds{P}[\widetilde{M}_k\geq k]\leq \kappa^k$ for all $k\geq 1$. Hence, we obtain the desired statement by combining this with \eqref{Mk-Mktilde}.
\end{proof}

With the above lemmas at hand, we can provide the

\begin{proof}[Proof of Lemma \ref{prob-bad}]
Recall that $M_k$ is defined in \eqref{def-Mk}. Then from it and Lemma \ref{lem-tail-Mk}, we arrive at
$$
w_k=\mathds{P}[\xi_{i_1}=0, \xi_{i_2}=0,\cdots,\xi_{i_k}=0]\leq \mathds{P}[M_k\geq k]\leq \kappa^k,
$$
where $\kappa$ is the constant defined in Lemma \ref{lem-tail-Mk}, depending only on $d$ and $\beta$.
\end{proof}

Combining Lemma \ref{prob-bad} with \cite[Theorem 3.4]{AA97} (see also \cite[Theorem 1]{RV10}), we can present the
\begin{proof}[Proof of Proposition \ref{numb-Di}]
Note that Lemma \ref{prob-bad} implies that conditions in \cite[Theorem 1]{RV10} are satisfied with $\delta=\kappa\in (0,1)$. Thus, by applying \cite[Theorem 1]{RV10} with $\delta=\kappa$ and $\gamma=(1+\kappa)/2$, we can conclude that
\begin{equation}\label{sum-xi}
\mathds{P}\left[\sum_{i=1}^K\xi_i\leq (1-\kappa)K/2\right]=\mathds{P}\left[\sum_{i=1}^K(1-\xi_i)\geq \gamma K\right]\leq \exp\{-KD(\gamma\|\delta)\},
\end{equation}
where $D(\gamma\|\delta):=\gamma \log\frac{\gamma}{\delta}+(1-\gamma)\log\frac{1-\gamma}{1-\delta}$ is the binary relative entropy between $\gamma$ and $\delta$. In addition, note that  $D(\gamma\|\delta)\geq 2(\gamma-\delta)^2=(1-\kappa)^2/2$. Plugging this into \eqref{sum-xi} yields the desired result.
\end{proof}

\subsection{Further events on metric regularities and long edges}
Recall that for $r>s>0$, $\mathbb{A}_{r,s}=B_r({\bm 0})\setminus B_s({\bm 0})$.
For $i\in [1,K]_{\mathds{Z}}$, let us denote
$$
D_i=D(\mathbb{A}_{r_{i-1}+M_i/8,r_{i-1}}, \mathbb{A}_{r_i,r_{i}-M_i/8}; \mathbb{A}_{r_i, r_{i-1}})
$$
as the internal $D$-distance between $\mathbb{A}_{r_{i-1}+M_i/8,r_{i-1}}$ and $ \mathbb{A}_{r_i,r_{i}-M_i/8}$ restricted to $\mathbb{A}_{r_i, r_{i-1}}$.
From the scale covariance of the LRP metric (see \cite[Theorem 1.4]{DFH23+}) and \eqref{di}, it is clear that
\begin{equation}\label{di-2}
\begin{aligned}
 \mathds{P}\left[D_i\geq c_{*,1} M_i^{\theta}\right]
&=\mathds{P}\left[M_i^\theta D(\mathbb{A}_{r_{i-1}/M_i+1/8,r_{i-1}/M_i}, \mathbb{A}_{r_i/M_i,r_i/M_i-1/8};\mathbb{A}_{r_i/M_i,r_{i-1}/M_i } )\geq c_{*,1} M_i^{\theta}\right]\\
&\geq \mathds{P}\left[D(B_{1/2}({\bm 0}), \mathbb{A}_{1,7/8};B_1({\bm 0}))\geq c_{*,1}\right]\geq c_{*,2},
\end{aligned}
\end{equation}
where $c_{*,1}, c_{*,2}>0$ (both depending only on $d$ and $\beta$) are the constants defined in \eqref{di}.

Recall that $\delta>0$ is the constant defined in \eqref{delta0}.
In the following, we divide $\mathbb{A}_{r_{i-1}+ M_i/8,r_{i-1}}$ and $ \mathbb{A}_{r_i,r_{i}-M_i/8}$ into $\delta M_i$ small cubes with side length $\delta M_i$, denoted by  $V^{(1)}_{i,1},\cdots, V^{(1)}_{i,s_{i,1}}$ and $V^{(2)}_{i,1},\cdots, V^{(2)}_{i,s_{i,2}}$, respectively. It can be observed that
\begin{equation}\label{s1s2}
\max\{s_{i,1},s_{i,2}\}\leq \delta^{-2d}.
\end{equation}

We first define an event on the regularity of the $D$-diameters of $V_{i,k}^{(j)}$ for all $j=1,2$ and all $k\in [1, s_{i,j}]_{\mathds{Z}}$.

\begin{definition}\label{def-Fi}
For each $i\in [1,K]_{\mathds{Z}}$, let $E_{i,2}$ be the event that $ \text{diam}(V^{(j)}_{i,k};D)\leq \delta^{\theta/2}M_i^{\theta}$ for all $j=1,2$ and all $k\in [1, s_{i,j}]_{\mathds{Z}}$.
\end{definition}


Our other events concern whether there are long edges directly connecting $V^{(1)}_{i,\cdot}$ and $V^{(2)}_{i,\cdot}$.

\begin{definition}\label{def-Ei3i4}
For fixed $i\in [1,K]_{\mathds{Z}}$,
\begin{itemize}

\item[(1)] let $E_{i,3}$ be the event that $ V^{(1)}_{i,k}\nsim V^{(2)}_{i,l}\ \text{for all }k\in [1,s_{i,1}]_{\mathds{Z}}$ and all $l\in [1, s_{i,2}]_{\mathds{Z}}$, and $D_i\geq c_{*,1} M_i^{\theta}$, where $c_{*,1}>0$ (depending only on $d$ and $\beta$) is the constant in \eqref{di}.
\medskip

\item[(2)] let $E_{i,4}$ be the event that $ V^{(1)}_{i,k}\sim V^{(2)}_{i,l}\ \text{for all }k\in [1,s_{i,1}]_{\mathds{Z}}$ and all $l\in [1, s_{i,2}]_{\mathds{Z}}$.
\end{itemize}
\end{definition}

\begin{remark}
	Later, which scenario in Definition~\ref{def-Ei3i4} occurs corresponds to a Bernoulli variable in the application of Sperner's theorem.
\end{remark}

In the following, we will estimate the probabilities of events $E_{i,2}$, $E_{i,3}$ and $E_{i,4}$.

\begin{lemma}\label{prop-Fi}
For all $d\geq 1$ and all $\beta>0$, there exists a constant $c_3=c_3(d,\beta)>0$ {\rm(}depending only on $d$ and $\beta${\rm)} such that for each $i\in [1,K]_{\mathds{Z}}$, we have
$$
\mathds{P}[E_{i,2}]\geq 1-c_3\delta^{-4d}\exp\left\{-\delta^{-\theta/2}\right\}=:p_{2}(\delta).
$$
\end{lemma}

\begin{proof}
According to the tightness of the distance (see \cite[Axiom V in Section 1.3]{DFH23+}), there exists a constant $\widetilde{c}_1=\widetilde{c}_1(d,\beta)>0$ (depending only on $d$ and $\beta$) such that for each $j=1,2$ and each $k=1,\cdots,s_{i,j}$,
\begin{equation*}
\begin{aligned}
&\mathds{P}\left[\text{diam}(V_{i,k}^{(j)};D)>\delta^{\theta/2}M_i^{\theta}\right]
=\mathds{P}\left[\frac{\text{diam}(V_{i,k}^{(j)};D)}{(\delta M_i)^{\theta}}>\delta^{-\theta/2}\right]\\
&\leq \exp\left\{-\delta^{-\theta/2}\right\}\mathds{E}\left[\exp\left\{\frac{\text{diam}(V_{i,k}^{(j)};D)}{(\delta M_i)^\theta}\right\}\right]
\leq \widetilde{c}_1 \exp\left\{-\delta^{-\theta/2}\right\}.
\end{aligned}
\end{equation*}
From this and \eqref{s1s2} we obtain that
$$
\mathds{P}[E_{i,2}]\geq 1-\widetilde{c}_1\delta^{-4d}\exp\left\{-\delta^{-\theta/2}\right\}.
$$
Hence, the proof is complete.
\end{proof}


\begin{lemma}\label{p3p4}
For all $d\geq 1$ and all $\beta>0$, there exist constants $c_4=c_4(d,\beta)>0$ and $c_5=c_5(d,\beta)>0$ {\rm(}both depending only on $d$ and $\beta${\rm)} such that for each $i\in [1,K]_{\mathds{Z}}$, we have
\begin{equation*}\label{pi1}
\mathds{P}[E_{i,3}]\geq c_{*,2}\exp\left\{-c_4\delta^{-2d}\right\}=:p_{3}(\delta)
\end{equation*}
and
\begin{equation*}\label{pi2}
\mathds{P}[E_{i,4}]\geq \left(1-\exp\left\{-c_5\delta^{2d}\right\}\right)^{\delta^{-4d}}:= p_{4}(\delta),
\end{equation*}
where $c_{*,2}>0$ {\rm(}depending only on $d$ and  $\beta${\rm)} is the constant defined in \eqref{di}.
\end{lemma}
\begin{proof}
Fix $i\in [1,K]_{\mathds{Z}}$. For each $k\in [1,s_{i,1}]_{\mathds{Z}}$ and $l\in [1,s_{i,2}]_{\mathds{Z}}$, since $\text{dist}(V^{(1)}_{i,k},V^{(2)}_{i,l})\geq 3M_i/4$, we have
\begin{equation}\label{VV}
\begin{aligned}
\mathds{P}\left[V^{(1)}_{i,k}\nsim V^{(2)}_{i,l}\right]&=\exp\left\{-\int_{V_{i,k}^{(1)}}\int_{V_{i,l}^{(2)}}\frac{\beta}{|{\bm u}-{\bm v}|^{2d}}\d {\bm u}\d {\bm v}\right\}\\
&\geq \exp\left\{-\int_{V_{i,k}^{(1)}}\int_{V_{i,l}^{(2)}}\frac{\beta}{(3/4)^{2d}M_i^{2d}}\d {\bm u}\d {\bm v}\right\}
\geq \exp\left\{-\widetilde{c}_1\delta^{2d}\right\}
\end{aligned}
\end{equation}
for some $\widetilde{c}_1=\widetilde{c}_1(d,\beta)>0$ depending only on $d$ and $\beta$. Combining this with \eqref{s1s2}, \eqref{di-2} and the independence of edges in the LRP model gives that
\begin{equation*}
\mathds{P}[E_{i,3}]\geq c_{*,2}\left(\exp\left\{-\widetilde{c}_1\delta^{2d}\right\}\right)^{\delta^{-4d}}
=c_{*,2}\exp\left\{-\widetilde{c}_1\delta^{-2d}\right\}.
\end{equation*}

In addition, it follows from $\text{dist}(V_{i,k}^{(1)}, V_{i,l}^{(2)})\leq r_i+r_{i-1}+\gamma M_i$ that
\begin{equation}\label{VV2}
\begin{aligned}
\mathds{P}\left[V^{(1)}_{i,k}\sim V^{(2)}_{i,l}\right]&=1-\exp\left\{-\int_{V_{i,k}^{(1)}}\int_{V_{i,l}^{(2)}}\frac{\beta}{|{\bm u}-{\bm v}|^{2d}}\d {\bm u}\d {\bm v}\right\}\\
&\geq 1-\exp\left\{-\int_{V_{i,k}^{(1)}}\int_{V_{i,l}^{(2)}}\frac{\beta}{(r_i+r_{i-1}+ \gamma M_i)^{2d}}\d {\bm u}\d {\bm v}\right\}=1-\exp\left\{\widetilde{c}_2\delta^{2d}\right\}
\end{aligned}
\end{equation}
for some $\widetilde{c}_2=\widetilde{c}_2(d,\beta)>0$ depending only on $d$ and $\beta$.
Combining \eqref{VV}, \eqref{VV2} with the independence of edges and \eqref{s1s2}, we complete the proof.
\end{proof}

Recall that $\kappa\in (0,1)$ is the constant (depending only on $d$ and $\beta$) in Proposition \ref{numb-Di}.
We now define
\begin{equation}\label{pdelta}
p(\delta)=(1-\kappa)p_2(\delta)(p_3(\delta)+p_4(\delta))/2
\end{equation}
and introduce the following event.

\begin{definition}\label{def-F}
Let $F$ be the event that there exists $I \subset [1, K]_{\mathbb Z}$ with $|I| \geq p(\delta) K$ such that $E_{i,1}^c\cap E_{i,2}\cap (E_{i,3}\cup E_{i,4})$ occurs for each $i\in I$.     
\end{definition}

With the above lemmas at hand, we can establish the following estimate for the event $F$.

\begin{proposition}\label{prob-AF}
For all $d\geq 1$ and all $\beta>0$, we have
$$
\mathds{P}[A^c\cap F]\geq 1-\left[c_1KN^{-d}+ \exp\left\{-(1-\kappa)^2K/2\right\}+ \exp\left\{-p(\delta)K/8\right\}\right],
$$
where $c_1,\ \kappa$ are constants {\rm(}both depending only on $d$ and $\beta${\rm)} in Lemma \ref{lem-probA} and Proposition \ref{numb-Di} respectively, and $p(\delta)$ {\rm(}depending only on $d, \beta$ and $\delta${\rm)} is the constant defined in \eqref{pdelta}.
\end{proposition}

\begin{proof}
Recall that $E_1$ is the event defined in Proposition \ref{numb-Di}. From Proposition \ref{numb-Di} we can see that
\begin{equation}\label{E1c}
\mathds{P}[E_1]\leq \exp\left\{-(1-\kappa)^2K/2\right\}.
\end{equation}
In addition, by the independence of edges, it can be observed that $\{E_{i,2}\cap (E_{i,3}\cup E_{i,4})\}_{i\in [1,K]_{\mathds{Z}}}$ and $E_1^c$ are all independent.  Therefore, from Chernoff bound and \eqref{E1c} we get that
\begin{equation*}
\begin{aligned}
\mathds{P}[F^c]&\leq \mathds{P}[E_1]+\mathds{P}[E_1^c\cap F^c]\\
&\leq \exp\left\{-(1-\kappa)^2K/2\right\}+ \exp\left\{-p(\delta)K/8\right\}.
\end{aligned}
\end{equation*}
Combining this with Lemma \ref{lem-probA}, we have
\begin{equation*}
\begin{aligned}
\mathds{P}[A\cup F^c]&\leq \mathds{P}[A]+\mathds{P}[F^c]\leq c_1 KN^{-d}+ \exp\left\{-(1-\kappa)^2K/2\right\}+ \exp\left\{-p(\delta)K/8\right\},
\end{aligned}
\end{equation*}
which implies the desired result.
\end{proof}

\subsection{Proof of Theorem \ref{thm-continuity}}

Recall that we fix a sufficiently small $\varepsilon \in (0,1)$ and fix ${\bm x}\in \mathds{S}$.
Throughout this subsection, we assume that the event $A^c\cap F$ occurs. Denote by $i_1,\cdots,i_J$ the subscripts such that  events $E_{i_j,1}^c\cap E_{i_j,2}\cap (E_{i_j,3}\cup E_{i_j,4}),\ j\in [1,J]_{\mathds{Z}}$ occur. Then from Definition \ref{def-F} we can see that $J\geq p(\delta)K$. Let
$$
\mathcal{E}_J=\bigcup_{j=1}^J\bigcup_{k=1}^{s_{i_j,1}}\bigcup_{l=1}^{s_{i_j,2}} \mathcal{E}_{V^{(1)}_{i_j,k}\times V^{(2)}_{i_j,l}},
$$
and denote
\begin{equation*}
\eta_{j}=
\begin{cases}
1,\quad & E_{i_j,4}\ \text{occurs};\\
0,\quad & E_{i_j,3}\ \text{occurs}
\end{cases}
\quad \quad \text{for all }j\in [1,J]_{\mathds{Z}}.
\end{equation*}
We can see that
$$
\mathds{P}\left[\eta_j=0\ |A^c\cap F, \mathcal{E}\setminus\mathcal{E}_J\right]+\mathds{P}\left[\eta_j=1\ |A^c\cap F, \mathcal{E}\setminus\mathcal{E}_J\right]=1.
$$
We also define $\mathcal{H}=\{j\in [1,J]_{\mathds{Z}}: \eta_j=1\}$. In the following analysis, we condition on the event $A^c \cap F$, on the value of $J$ (recall the definition of $J$ from above) and on $\mathcal E\setminus \mathcal E_J$. Therefore, the only remaining randomness is on the value of $\{\eta_j: j\in J\}$. For fixed $a>0$, let $G_{a,\varepsilon}$ denote the event (i.e., a collection of realizations for $\{\eta_j: j\in J\}$) that $D({\bm 0},{\bm x})\in (a-\varepsilon, a+\varepsilon)$.

\begin{lemma}\label{G-sperner}
For fixed $\varepsilon \in (0,1)$, recall that $\{M_i\}_{i\in [1,K]_{\mathds{Z}}}$ is a sequence defined in \eqref{def-M}.
For all $d\geq 1$ and all $\beta>0$, there exists sufficiently small $\delta\in (0,1)$ {\rm(}depending only on $d$ and $\beta${\rm)} such that condition on $A^c\cap F$, on the value of $J$ and on $\mathcal{E}\setminus\mathcal{E}_J$, $G_{a,\varepsilon}$ is a Sperner family for each $a>0$.
\end{lemma}

\begin{proof}
Fix the conditioning as in the lemma statement throughout the proof.
Let us denote $P$ as a $D$-geodesic from ${\bm 0}$ to ${\bm x}$. Assume that $H$ is a realization of $\mathcal{H}$ such that on the event $\mathcal{H}=H$, $G_{a,\varepsilon}$ occurs, i.e.,
\begin{equation}\label{lenP}
\text{len}(P;D)=D({\bm 0},{\bm x})\in (a-\varepsilon,a+\varepsilon).
\end{equation}

We now take $j'\in [1,J]_{\mathds{Z}}$ such that $j'\notin H$ and  define $H'=H\cup \{j'\}$.
We claim that on the event $\mathcal{H}=H'$,  it must hold that $D({\bm 0},{\bm x})\notin (a-\varepsilon,a+\varepsilon)$.
Indeed, on the event $\mathcal{H}=H'$, we can construct a new path $P'$ from ${\bm 0}$ to ${\bm x}$ as follows.
Let ${\bm x}_{j',1}$ and ${\bm x}_{j',2}$ be the points where the path $P$ enters into $\mathbb{A}_{r_{i_{j'}-1}+ M_{i_{j'}}/8, r_{i_{j'}-1}}$ and $\mathbb{A}_{r_{i_{j'}}, r_{i_{j'}}-M_{i_{j'}}/8}$, respectively.
Let $k',l'$ be indices such that  ${\bm x}_{j',1}\in V_{i_{j'},k'}^{(1)}$ and ${\bm x}_{j',2}\in V_{i_{j'},l'}^{(2)}$.
Since $j'\in \mathcal{H}=H'$, this implies that $E_{i_{j'},4}$ occurs, implying that there exists an edge
$$
\langle {\bm z}_1,{\bm z}_2 \rangle \in \mathcal{E}_{V_{i_{j'},k'}^{(1)}\times V_{i_{j'},l'}^{(2)}}.
$$
Therefore, we let $P'({\bm x}_{j',1},{\bm x}_{j',2})$ be the path starting at ${\bm x}_{j',1}$, which takes the shortest path (restricted to $V_{i_{j'},k'}^{(1)}$) to ${\bm z}_1$, then uses the edge $\langle {\bm z}_1,{\bm z}_2 \rangle$ to arrive ${\bm z}_2$, and finally takes the shortest path (restricted to $V_{i_{j'},l'}^{(2)}$) to ${\bm x}_{j',2}$. Then the path $P'$ is defined by replacing the part of path $P$ from  ${\bm x}_{j',1}$ to ${\bm x}_{j',2}$  with $P'({\bm x}_{j',1},{\bm x}_{j',2})$. From \eqref{delta0} and \eqref{def-M} we obtain that
\begin{equation*}
\begin{aligned}
\text{len}(P;D)-\text{len}(P';D)
&\geq c_{*,1} M_{i_{j'}}^\theta- \text{diam}(V_{i_{j'},k'}^{(1)};D)-\text{diam}(V_{i_{j'},l'}^{(2)};D)\\
&\geq c_{*,1} M_{i_{j'}}^\theta-2\delta^{\theta/2}M_{i_{j'}}^\theta
\geq \frac{1}{2}c_{*,1} M_{i_{j'}}^\theta\geq \frac{1}{2}c_{*,1} M_1^\theta= 2\varepsilon.
\end{aligned}
\end{equation*}
Combining this with \eqref{lenP}, we get that $\text{len}(P';D)<a-\varepsilon$, which implies
\begin{equation*}
D({\bm 0}, {\bm x})\leq \text{len}(P';D)\notin (a-\varepsilon, a+\varepsilon).
\end{equation*}
Hence, the claim holds.
Using similar arguments, we can also show that for any $j'\in H$, on the event $\mathcal{H}=H\setminus \{j'\}$ it must hold that $D(\bm 0, \bm x)\notin(a-\varepsilon,a+\varepsilon)$.
From this and the definition of Sperner family in Definition \ref{def-sperner}, we get that condition on $A^c\cap F$, on the value of $J$ and $\mathcal{E}\setminus\mathcal{E}_J$, $G_{a,\varepsilon}$ is a Sperner family for each $a>0$.
\end{proof}

In the rest of this section, we present the

\begin{proof}[Proof of Theorem \ref{thm-continuity}]
Fix a sufficiently small $\varepsilon\in (0,1)$  and $a>0$.
Recall that $K$, $\{M_i\}_{i\in [1,K]_{\mathds{Z}}}$ and $\delta$ are the constants defined in \eqref{def-KN}, \eqref{def-M} and \eqref{delta0}, respectively.

Condition on $A^c\cap F$, on the value of $J$ and on $\mathcal{E}\setminus\mathcal{E}_J$, we have $\{\eta_j:j\in[1,J]\}$ is a series of i.i.d.\ Bernoulli variables with parameter $p_\eta:=\mathds{P}[E_{i_j,4}|E_{i_j,4}\cup E_{i_j,3}]$ from the independence between $\{E_{i,2}\}_{i\in [1,K]_{\mathds{Z}}},\{E_{i,3}, E_{i,4}\}_{i\in [1,K]_{\mathds{Z}}}$ and $\{A^c,E_{i,1}\}_{i\in[1,K]_\mathds{Z}}$. Additionally, note that from Lemma~\ref{p3p4}, we have for each $j\in[1,J]_{\mathds{Z}}$,
\begin{equation*}
	p_\eta\in\left[\mathbb{P}[E_{i_j,4}],1-\mathbb{P}[E_{i_j,3}]\right]\subset[p_4(\delta),1-p_3(\delta)]. 
\end{equation*}
Thus combining this with Lemma \ref{G-sperner} and Proposition \ref{prop-sperner}, we have
\begin{equation*}
\mathds{P}[G_{a,\varepsilon}\ |A^c\cap F, \mathcal{E}\setminus\mathcal{E}_J]\leq \frac{C_1}{\sqrt{p(\delta)K}}
\end{equation*}
for some constant $C_1=C_1(d,\beta)>0$ depending only on $d$ and $\beta$. This implies that $\mathds{P}[G_{a,\varepsilon}\ |A^c\cap F]\leq \frac{C_1}{\sqrt{p(\delta)K}}$.
Combining this with Lemma \ref{prob-AF}, we have
\begin{equation*}\label{PG-sperner}
\begin{aligned}
\mathds{P}[G_{a,\varepsilon}]&\leq \mathds{P}[A\cup F^c]+\mathds{P}[G_{a,\varepsilon}\ |A^c\cap F]\\
&\leq c_1KN^{-d}+ \exp\left\{-(1-\kappa)^2K/2\right\}+ \exp\left\{-p(\delta)K/8\right\}+\frac{C_1}{\sqrt{p(\delta)K}},
\end{aligned}
\end{equation*}
where $c_1, \kappa,\ p(\delta)$ are constants (both depending only on $d$ and $\beta$) in Lemma \ref{lem-probA}, Proposition \ref{numb-Di} and \eqref{pdelta}, respectively. Furthermore, by \eqref{def-KN}, we obtain that
$$
\mathds{P}[G_{a,\varepsilon}]\leq \widetilde{c}_1 (\log(1/\varepsilon))^{-1/4}
$$
for some constant $\widetilde{c}_1=\widetilde{c}_1(d,\beta)$ depending only on $d$ and $\beta$. Hence, we have $\lim_{\varepsilon\to 0^+}\mathds{P}[G_{a,\varepsilon}]=0$ for all $a>0$.
\end{proof}

\section{Proof of Theorem \ref{thm-unique}}\label{sect-T2}

In this section, we fix two distinct points $\bm x,\bm y\in \mathds{R}^d$.  To prove Theorem \ref{thm-unique}, we start with some notations.
For any $q\in \mathds{Q}_+$, denote
$$
\mathcal{E}_{\geq q}=\{\langle \bm u,\bm v\rangle\in\mathcal{E}:|\bm u-\bm v|\geq q\}.
$$
In the following, we fix a sufficiently small $q\in \mathds{Q}_+$ and condition on the edge set $\mathcal{E}_{\geq q}$.
We select a long edge $\langle \bm u,\bm v\rangle\in \mathcal{E}_{\geq q}$ that satisfies $\bm u\neq \bm x,\bm y$ and $\bm v\neq \bm x,\bm y$.
Additionally, we choose a sufficiently small $r\in (0,q)\cap \mathds{Q}$ such that
\begin{itemize}

\item[(R1)]$\bm x,\bm y\notin B_r(\bm v)$;
\medskip

\item[(R2)] $B_r(\bm v)$ contains no other endpoints of edges in $\mathcal{E}_{\geq q}$ except point $\bm v$.
\end{itemize}
We also assume that $(\mathcal{E}\setminus\mathcal{E}_{\geq q})\cap (B_r(\bm v)^c\times \mathds{R}^d)$ is given. Note that, by the above assumption, the edge set
\begin{equation}\label{given-es}
\mathcal{E}({\bm v}):=\mathcal{E}_{\geq q}\bigcup\left((\mathcal{E}\setminus\mathcal{E}_{\geq q})\bigcap (B_r(\bm v)^c\times \mathds{R}^d)\right)
\end{equation}
is given.
For convenience, denote
\begin{equation*}\label{def-Z}
Z_r(\bm v)=\left\{\bm z\in B_r(\bm v)\setminus\{\bm v\}: \exists \bm w\in B_r(\bm v)^c \text{ such that }\langle \bm z,\bm w\rangle\in \mathcal{E}\right\} \cup(\partial B_r(\bm v)\cap \mathds{Q}).
\end{equation*}
It is clear that $Z_r(\bm v)\subset \mathds{R}^d$ is a countable set by the property of Poisson point process. Moreover, since the edge set in \eqref{given-es} is given, $Z_r(\bm v)$ is a deterministic set.

\begin{definition}\label{def-Quvz}
For any $\bm z\in Z_r(\bm v)$, we will construct a path $Q_{\bm u\bm v\bm z}(\bm x,\bm y)$ from $\bm x$ to $\bm y$ as outlined below. Starting at $\bm x$, $Q_{\bm u\bm v\bm z}(\bm x,\bm y)$ proceeds as follows: it first follows a $D$-geodesic from $\bm x$ to $\bm u$ restricted to $B_r(\bm v)^c$, then traverses the long edge $\langle \bm u,\bm v\rangle$ to reach $\bm v$, and next it follows a $D$-geodesic from $\bm v$ to $\bm z$ restricted to $B_r(\bm v)$, 
and finally it uses a $D$-geodesic from $\bm z$ to $\bm y$ restricted to $B_r(\bm v)^c$. See Figure \ref{fig-Quvz} for an illustration. It is worth emphasizing that we will choose oly one geodesic in each step in a deterministic way and as a result there is only one $Q_{\bm u \bm v \bm z}(\bm x,\bm y)$ for any $\bm z\in Z_r(\bm v)$.
\end{definition}

\begin{figure}\label{fig-Quvz}
	\centering
	\includegraphics{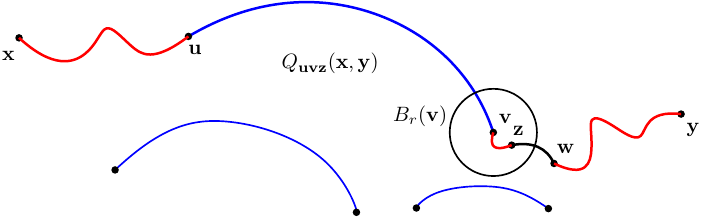}
	\caption{Illustration for Definition \ref{def-Quvz}. The blue lines represent the edges in $\mathcal{E}_{\geq q}$ including $\langle \bm u, \bm v\rangle$. The black line represents the edge $\langle \bm z,\bm w\rangle$ with $\bm z\in Z_r(\bm v)$. The red lines represent three $D$-geodesics taken in Definition \ref{def-Quvz}.}
\label{fig-Quvz}
\end{figure}

\begin{lemma}\label{prob-z}
For fixed sufficiently small $q\in \mathds{Q}_+$, assume that the edge set in \eqref{given-es} is given. Then for any $\langle \bm u,\bm v\rangle\in \mathcal{E}_{\geq q}$ with $\bm u\neq \bm x,\bm y$ and $\bm v\neq \bm x,\bm y$, and any sufficiently small $r\in (0,q)\cap \mathds{Q}$ satisfying {\rm(R1)} and {\rm(R2)}, the following holds. For any $\bm z\in Z_r(\bm v)$,  we have that almost surely,  ${\rm len}(Q_{\bm u\bm v\bm z}(\bm x,\bm y);D)\neq D(\bm x,\bm y;B_r(\bm v)^c)$.
\end{lemma}

\begin{proof}
For convenience, denote $a=D(\bm x,\bm y;B_r(\bm v)^c)$ and $b=\text{len}(Q_{\bm u\bm v\bm z}(\bm x,\bm y)\cap B_r(\bm v)^c;D)$. It is worth emphasizing that, since the edge set $\mathcal{E}(\bm v)$ defined in \eqref{given-es} is given, $a$ and $b$ are deterministic constants.  Therefore, it suffices to prove that
\begin{equation}\label{prob-Q}
\mathds{P}\left[\text{len}(Q_{\bm u\bm v\bm z}(\bm x,\bm y)\cap B_r(\bm v);D)=a-b\ |\mathcal{E}(\bm v)\right]=0.
\end{equation}
It is obvious that \eqref{prob-Q} holds trivially when $b\geq a$. In the rest of the proof, we assume that $b<a$. We now turn to show that
\begin{equation}\label{prob-Q2}
\lim_{\varepsilon\to 0^+}\mathds{P}\left[\text{len}(Q_{\bm u\bm v\bm z}(\bm x,\bm y)\cap B_r (\bm v);D)\in (a-b-\varepsilon,a-b+\varepsilon)\ |\mathcal{E}(\bm v)\right]=0,
\end{equation}
which implies \eqref{prob-Q} directly.  Indeed, according to the proof of Theorem \ref{thm-continuity} by replacing $\bm 0$, $\bm x$ and  $M_i$ with $\bm v$, $\bm z$ and $|\bm v-\bm z|M_i$, respectively, we can obtain \eqref{prob-Q2} immediately.
\end{proof}

We now turn to the

\begin{proof}[Proof of Theorem \ref{thm-unique}]
Let $\bm x$ and $\bm y$ be two distinct points in $\mathds{R}^d$.
Fix $q\in \mathds{Q}_+$ and assume that $\mathcal{E}_{\geq q}$ is given. For convenience, we let
$$
\mathds{Q}_q=\left\{r\in (0,q)\cap \mathds{Q}:\ r\ \text{satisfies (R1) and (R2)}\right\}.
$$
Now for any $r\in \mathds{Q}_q$, $\langle \bm u,\bm v \rangle\in \mathcal{E}_{\geq q}$ and any $\bm z\in Z_r(\bm v)$, recall that the path $Q_{\bm u \bm v\bm z}(\bm x,\bm y)$ is defined in Definition \ref{def-Quvz}.
Denote
\begin{equation}\label{def-Puv}
\mathcal{P}_{\bm u\bm v}=\left\{Q_{\bm u \bm v\bm z}(\bm x,\bm y):\ r\in \mathds{Q}_q,\ z\in Z_r(\bm v)\right\}.
\end{equation}
Then according to Lemma \ref{prob-z} and the fact that $\mathds{Q}_q$ and $Z_r(\bm v)$ for $r\in \mathds{Q}_q$ are all countable sets, we can obtain that condition on $\mathcal{E}(\bm v)$, almost surely
\begin{equation}\label{prob-r}
\text{len}(Q;D)\neq D(\bm x,\bm y;B_r(\bm v)^c)\quad \text{for all }Q\in \mathcal{P}_{\bm u\bm v}.
\end{equation}
Furthermore, since there are at most countable edges in $\mathcal{E}$ and $q\in \mathds{Q}_+$, from \eqref{prob-r} we arrive at
\begin{equation}\label{lenQ=D}
\mathds{P}\left[\text{len}(Q;D)=D(\bm x,\bm y;\mathcal{E}\backslash \{\langle \bm u,\bm v\rangle\})\ \text{for some }\langle \bm u,\bm v\rangle\in \mathcal{E}\ \text{and some }Q\in \mathcal{P}_{\bm u\bm v}\right]=0.
\end{equation}

Additionally, note that if there are two distinct $D$-geodesics from $\bm x$ to $\bm y$, Denote by $P_1$ and $P_2$, then there must exist an edge $\langle \bm u,\bm v\rangle \in \mathcal{E}$ that is traversed by $P_1$ and not by $P_2$.
Combining this observation with \eqref{def-Puv} and Definition \ref{def-Quvz}, we can conclude that the event \{there are multiple $D$-geodesics from $\bm x$ to $\bm y$\} implies the event in \eqref{lenQ=D}. Hence,
we obtain the desire result from this fact and \eqref{lenQ=D}.
\end{proof}

\section{Proof of Theorem \ref{mr-dim}}\label{sect-dim}

In this section,  we present the proof of Theorem~\ref{mr-dim}. We will establish the lower bound for the Hausdorff dimension of $D$-geodesics in  Section~\ref{sect-dim-low}, followed by the upper bound in Section~\ref{sect-dim-up}. To this end, for any metric $\widetilde{D}$, denote by $\text{diam}(\cdot;\widetilde{D})$ the diameter of $\widetilde{D}$.

\subsection{Lower bound}\label{sect-dim-low}
The main tool in proving the lower bound for the Hausdorff dimension of $D$-geodesics  is the following mass distribution principle.
\begin{lemma}[{\cite[Theorem~4.19]{Morters_Peres_2010}}]\label{mass}
Recall \eqref{hausdorff}. Suppose that $(E,\widehat{D})$ is a metric space and $\Delta\geq 0$. If there exists a mass distribution $\zeta$ on $E$ along with constants $C <\infty$ and $\delta > 0$ such that
\begin{equation*}
	\zeta(V)\leq C\left({\rm diam}(V;\widehat{D})\right)^\Delta,
\end{equation*}
for all closed sets $V$ with ${\rm diam}(V;\widehat{D})\leq \delta$, then
\begin{equation*}
	C_\Delta(E,\widehat{D})\geq C^{-1}\zeta(E)>0
\end{equation*}
and hence $\mathrm{dim}_{\mathrm{H}}(E;\widehat{D})\geq \Delta$.
\end{lemma}

We also require the following lemma to characterize the uniform continuity of the metric $D$, which can be regarded as a corollary of \cite[Proposition~1.13]{DFH23+}. To introduce this, for each $L>0,\ \varepsilon>0$ and $k\in\mathds{N}$, denote $F_{k,L,\varepsilon}$ as the event that $D(\bm u,\bm v)\leq 2^{-(\theta-\varepsilon)k}$ for all ${\bm u},{\bm v}\in[-L,L]^d$ with $\|{\bm u}-{\bm v}\|_{\ell^\infty}\leq 2^{-k}$, where $\theta\in (0,1)$ is the exponent defined in \eqref{def-theta}.
\begin{lemma}\label{F-low}
For all $d\geq 1$, $\beta>0$ and $L>0$, there exist constants $c_1, c_2>0$ {\rm(}depending on $d, \beta$ and $L${\rm)} such that for each $k\in\mathds{N}$ and $\varepsilon>0$,
\begin{equation*}
\mathds{P}[F_{k,L,\varepsilon}^c]\leq c_1\exp\left\{-c_2 2^{\varepsilon k}/k\right\}.
\end{equation*}
Furthermore, for each fixed $L>0$ and $\varepsilon>0$, almost surely, there exists a {\rm(}random{\rm)} $k_0\in\mathds{N}$ such that $\cap_{k>k_0}F_{k,L,\varepsilon}$ occurs. 
\end{lemma}

\begin{proof}
	From \cite[Proposition~1.13]{DFH23+}, there exist constants $\widetilde{c}_1, \widetilde{c}_2>0$ (depending only on $d$ and $\beta$) such that for each $L>0, M>0$ and $k\in\mathds{N}$,
	\begin{equation}\label{cont}
		\mathds{P}\left[\max_{{\bm u},{\bm v}\in[-L,L]^d}\frac{D({\bm u},{\bm v})}{\|{\bm u}-{\bm v}\|_{\ell^\infty}^\theta \log(\frac{4L}{\|{\bm u}-{\bm v}\|_{\ell^\infty}})}>M\right]\leq \widetilde{c}_1\exp\{-\widetilde{c}_2 M\}.
	\end{equation}

We now consider the event $F_{k,L,\varepsilon}^c$. Note that if $F_{k,L,\varepsilon}^c$ occurs, then there is a pair of ${\bm u},{\bm v}\in[-L,L]^d$ such that $\|{\bm u}-{\bm v}\|_{\ell^\infty}\leq 2^{-k}$ and $D({\bm u},{\bm v})>2^{-(\theta-\varepsilon)k}$. Therefore, for this pair of $({\bm u},{\bm v})$ we have
	\begin{equation*}\label{discont}
		\frac{D({\bm u},{\bm v})}{\|{\bm u}-{\bm v}\|_{\ell^\infty}^\theta\log(\frac{4L}{\|{\bm u}-{\bm v}\|_{\ell^\infty}})}>\frac{2^{-\theta k+\varepsilon k}}{\max_{0<t\leq 2^{-k}}t^\theta \log(4L/t)}\geq \widetilde{c}_3 2^{\varepsilon k}/k
	\end{equation*}
	for some $\widetilde{c}_3>0$ depending only on $d,\beta$ and $L$. This implies that the event on the LHS of \eqref{cont} occurs with $M=\widetilde{c}_3 2^{\varepsilon k}/k$. Thus we get that
	\begin{equation*}
		\mathds{P}[F_{k,L,\varepsilon}^c]\leq \widetilde{c}_1\exp\left\{-\widetilde{c}_2\widetilde{c}_3 2^{\varepsilon k}/k\right\},
	\end{equation*}
which implies the first statement. Furthermore, we can obtain the second statement by Borel-Cantelli Lemma.
\end{proof}

With Lemma~\ref{F-low} at hand, we now give the
\begin{proof}[Proof of lower bound in Theorem~\ref{mr-dim}]
	For fixed $L>0$ and $\varepsilon>0$, from Lemma~\ref{F-low}, it follows that almost surely, there exists a (random) $k_0\in \mathds{N}$ such that the event $\cap_{k\geq k_0} F_{k,L,\varepsilon}$ occurs.

	For each $D$-geodesic $P$ whose range is contained in $[-L,L]^d$, we use the $D$-distance to parametrize $P$, that is, $P: [0,D(\bm x, \bm y)]\to [-L,L]^d$.
It is evident that $D(P(s),P(t))=|s-t|$ for all $s,t\in[0,D(\bm x, \bm y)]$.

	Now we define a mass distribution $\zeta_P$ of $P$ as follows. For each $A\subset \mathrm{Range}(P)$, let
	\begin{equation*}
		\zeta_P(A)=\mathrm{Leb}\left(\{t\in[0,D(\bm x, \bm y)]:P(t)\in A\}\right).
	\end{equation*}
	We claim that almost surely, for each $D$-geodesic $P\subset [-L,L]^d$ and each closed subset $V\subset P$ with $|V|<2^{-k_0}$, we have
	$$\zeta_P(V)\leq 2^{\theta-\varepsilon}|V|^{\theta-\varepsilon}.$$
	Indeed, we will prove a slightly strong version: almost surely, for each $D$-geodesic $P\subset[-L,L]^d$ and each closed set $V\subset P$ with $|V|<2^{-k}$ for some $k\geq k_0$, we have
	$$\zeta_P(V)\leq 2^{-(\theta-\varepsilon)k}.$$
To prove this claim, we proceed by an contradiction.
Assume that $\zeta_P(V)>2^{-(\theta-\varepsilon)k}$ for some $V\subset P$ with $|V|\leq 2^{-k}$ and $k\geq k_0$.
Then we can see that there exist $s,t\in[0,\mathrm{len}(P)]$ such that $|s-t|\geq 2^{-(\theta-\varepsilon)k}$, which implies that $D(P(s),P(t))\geq 2^{-(\theta-\varepsilon)k}$.
However, since  $P(s),P(t)\in V$ and $|V|\leq 2^{-k}$, we have $\|P(s)-P(t)\|_{\ell^\infty}\leq 2^{-k}$. Thus, we conclude that $F_{k,L,\varepsilon}^c$ must occur for some $k\geq k_0$.
This contradicts our earlier statement in the first paragraph, leading to the desired claim.

From the claim and Lemma~\ref{mass} we get that almost surely for each $P\subset[-L,L]^d$ (here for each geodesic $P$, we apply Lemma \ref{mass} with $E = P$ and $\widehat D$ being Euclidean distance),
$$\mathrm{dim}_{\mathrm{H}}(\mathrm{Range}(P))\geq \theta-\varepsilon.$$
taking $\varepsilon\to 0$ and then allowing $L\to\infty$, we arrive at the desired result.
\end{proof}

\subsection{Upper bound}\label{sect-dim-up}

The proof of the upper bound in Theorem~\ref{mr-dim} primarily relies on a renormalization argument that is similar to the approach used in \cite[Section~3]{DFH23+}.
To this end, we begin with some notations which will be used repeatedly. For any $r>0$ and $\bm k\in \mathds{Z}^d$, let $V_r(\bm k)$ denote the cube in $\mathds{R}^d$ centered at $\bm k$ with side length $r$.
In the following, fix $L>0$ (which will eventually tend to infinity) and $k\in\mathds{N}$ with $\delta=2^{-k}\ll 1$.
We now divide $\mathds{R}^d$  into small cubes of side length $\delta L$, denoted by $V_{\delta L}({\bm k})$ for ${\bm k}\in(\delta L)\mathds{Z}^d$.
Then we identify the cubes $V_{\delta L}({\bm k})$ as vertices ${\bm k}$ and call the resulting graph $G$.
By the self-similarity of the model, it is obvious that $G$ is also the critical long-range percolation model on $(\delta L)\mathds{Z}^d$.

We introduce the definitions of ``good'' cubes and ``good'' vertices as follows.

\begin{definition}\label{h-good}
For $s>0$, ${\bm z}\in\mathds{R}^d$, $\alpha\in(0,1)$ and  $b>0$, we say that a cube $V_{3s}({\bm z})$ is $(3s,\alpha, b)$-good if the following condition holds.
For any two different edges $\langle {\bm u}_i,{\bm v}_i\rangle\in\mathcal{E}$, $i=1,2$, with ${\bm u}_1\in V_s({\bm z})^c$, ${\bm v}_1\in V_s({\bm z})$, ${\bm u}_2\in V_{3s}({\bm z})$ and ${\bm v}_2\in V_{3s}({\bm z})^c$, we have $|{\bm v}_1-{\bm u}_2|\geq \alpha s$ and
$$
D({\bm v}_1,{\bm u}_2;V_{3s}({\bm z}))\geq (b\alpha s)^\theta.
$$
Additionally, we say that the tuple $({\bm u}_1,{\bm v}_1, {\bm u}_2,{\bm v}_2)$ forms a $V_{3s}({\bm z})$-special pair of edges if $\langle {\bm u}_i,{\bm v}_i\rangle\in\mathcal{E}$, $i=1,2$, with ${\bm u}_1\in V_s({\bm z})^c$, ${\bm v}_1\in V_s({\bm z})$, ${\bm u}_2\in V_{3s}({\bm z})$ and ${\bm v}_2\in V_{3s}({\bm z})^c$.
\end{definition}

We say that ${\bm k}\in (\delta L)\mathds{Z}^d$ is good if $V_{3\delta L}({\bm k})$ is $(3\delta L,\alpha, b)$-good.
The following lemma from \cite{DFH23+} shows that for suitable choices of $\alpha$ and $b$, ${\bm k}$ is good with arbitrarily high probability.

\begin{lemma}[{\cite[Corollary~3.8]{DFH23+}}]\label{probgood}
	For all $d\geq 1$, $\beta>0$ and all $p_0\in(0,1)$, there exist constants $\alpha,\ b>0$ {\rm(}depending only on $d,\beta$ and $p_0${\rm)} such that
	\begin{equation*}
		\mathds{P}[\text{${\bm k}$ is good}]=\mathds{P}[\text{$V_{3\delta L}({\bm k})$ is $(3\delta L, \alpha, b)$-good}]\geq 1-p_0.
	\end{equation*}
\end{lemma}

Now, we turn our attention to paths in the original LRP model and the boxes that are intersected by such paths.
For each path $P\subset[-L,L]^d$ in the $\beta$-LRP model (including $D$-geodesics), we define
\begin{equation*}
	\Lambda_{\delta,L}(P)=\left\{{\bm k}\in(\delta L)\mathds{Z}^d:P\cap V_{\delta L}({\bm k})\neq \emptyset\right\}\quad\text{and}\quad  N_{\delta,L}(P)=\# \Lambda_{\delta,L}(P).
\end{equation*}
Additionally, for any subset $\Lambda\subset G$ in the renormalized graph $G$, let $\mathcal{C}_{\delta,L}(\Lambda)$ be the set of paths $P\subset [-L,L]^d$ in the original model such that $\Lambda_{\delta,L}(P)=\Lambda$.

To prove the upper bound on the Hausdorff dimension of $P$, we need to provide an upper bound for $N_{\delta,L}(P)$.
It is important to note that $\Lambda_{\delta,L}(P)$ forms a connected subgraph of $G$ by the renormalization and the fact that $P$ is connected.
Therefore, we need to derive some estimates for a fixed connected subset.
To achieve this, we give the following notion of a ``good'' subset, which can be viewed as a generalization of \cite[Definition~3.12]{DFH23+}.
\begin{definition}\label{def-good-tree}
	For any $\alpha, b>0$ and  any connected set $\Lambda=\{{\bm k}_1,\cdots,{\bm k}_n\}\subset G$, we say that $\Lambda$ is $(\alpha,b)$-good if at least $n/(2\cdot 3^d)$ of its indices $j\in[1,n]_\mathds{Z}$ satisfy the following conditions:
	\begin{enumerate}
		\item[(1)] For each $P\in\mathcal{C}_{\delta,L}(\Lambda)$, if $({\bm u}_{{\bm k}_j,1}, {\bm v}_{{\bm k}_j,1},{\bm u}_{{\bm k}_j,2},{\bm v}_{{\bm k}_j,2})$ is a $V_{3\delta L}({\bm k}_j)$-special pair of edges such that $P$ first uses $\langle {\bm u}_{{\bm k}_j,1}, {\bm v}_{{\bm k}_j,1}\rangle$ to enter $V_{\delta L}({\bm k}_j)$ and then does not leave $V_{3\delta L}({\bm k}_j)$ until it uses $\langle {\bm u}_{{\bm k}_j,2}, {\bm v}_{{\bm k}_j,2}\rangle$, then it satisfies
		\begin{equation*}
			|{\bm v}_{{\bm k}_j,1}- {\bm u}_{{\bm k}_j,2}|\geq \alpha\delta L \quad \text{and}\quad  D({\bm v}_{{\bm k}_j,1}, {\bm u}_{{\bm k}_j,2}; V_{3\delta L}({\bm k}_j)) \geq (b\alpha \delta L)^\theta.	
		\end{equation*}
		
		\item[(2)] For any two distinct indices ${j}_1$ and ${ j}_2$, we have ${\bm k}_{{ j}_1} - {\bm k}_{{ j}_2} \in (3\delta L)\mathds{Z}^d$.
	\end{enumerate}
	If $\Lambda$ is not $(\alpha,b)$-good, then we say that $\Lambda$ is $(\alpha,b)$-bad.
\end{definition}

The following lemma is a slightly stronger version  of \cite[Lemma~3.13]{DFH23+}.

\begin{lemma}\label{good-tree}
	For any $n\in \mathds{N}$ and $p>0$, let $\alpha$ and $b $ be chosen from Lemma~\ref{probgood} with $p_0=p$.
Then for any connected set $\Lambda=\{{\bm k}_1,\cdots,{\bm k}_n\}\subset G$, the probability that $\Lambda$ is $(\alpha,b)$-good is at least $1-(2p^{1/(2\cdot 3^d)})^n$.
\end{lemma}
\begin{proof}
	For any ${\bm i}\in(\delta L)\{0,1,2\}^d$, let $\Lambda^{(\bm i)}=\{{\bm k}\in\Lambda:{\bm k}-{\bm i}\in (3\delta L)\mathds{Z}^d\}$. Then the collection $\{\Lambda^{({\bm i})}:{\bm i}\in(\delta L)\mathds{Z}^d\}$ forms a partition of $\Lambda$. By the Pigeonhole's Principle, we can find an ${\bm i}^*\in(\delta L)\{0,1,2\}^d$ such that $|\Lambda^{({\bm i}^*)}|\geq N/3^d$. Note that the sets $V_{3\delta L}({\bm k})$ for all ${\bm k} \in \Lambda({\bm i}^*)$ are disjoint.

Next, we assume that $\Lambda$ is $(\alpha,b)$-bad. Then by the Definition \ref{def-good-tree}, it is clear that there must exist $P\in\mathcal{C}_{\delta,L}(\Lambda)$ and at least $n/(2\cdot 3^d)$ indices ${\bm k}$'s in $\Lambda^{({\bm i}^*)}$ such that for these corresponding $V_{3\delta L}(\bm k)$-special pairs of edges $({\bm u}_{{\bm k},1}, {\bm v}_{{\bm k},1},{\bm u}_{{\bm k},2},{\bm v}_{{\bm k},2})$, we have either
	\begin{equation*}
		\text{ $|{\bm v}_{{\bm k},1}-{\bm u}_{{\bm k},2}|<\alpha \delta L$\quad or\quad  $D({\bm v}_{{\bm k},1},{\bm u}_{{\bm k},2}; V_{3\delta L}(\bm k))<(b\alpha\delta L)^\theta$}.
	\end{equation*}
	This implies that there are at least $n/(2\cdot 3^d)$ indices ${\bm k}$ in $\Lambda^{({\bm i}^*)}$ such that $V_{3\delta L}(\bm k)$ is not $(3\delta L,\alpha,b)$-good, with disjoint witnesses. Therefore, from Lemma~\ref{probgood} and the BK inequality \cite{BK85}, we obtain
	\begin{equation*}
		\mathds{P}\left[\Lambda\text{ is $(\alpha, b)$-bad}\right]\leq 2^np^{n/(2\cdot 3^d)}.
	\end{equation*}
	This completes the proof.
\end{proof}

We also need the following estimate for the number of possible connected subset $\Lambda\subset G$, as stated in \cite[Lemma~3.2, (24)]{Ba23}.
For $k\in\mathds{N}$ and ${\bm i}\in(\delta L)\mathds{Z}^d$, let $\mathcal{CS}_k({\bm i})$ denote the collection of all connected subsets of the graph $G$ that have size $k$ and include the vertex ${\bm i}$. Define $\text{deg}(\bm i)$ as the degree of $\bm i$ in the graph $G$, that is,
$\text{deg}(\bm i)=\#\{\bm j\in (\delta L)\mathds{Z}^d:\ \bm j\sim \bm i\}$.

\begin{lemma}[{\cite[Lemma~3.2, (24)]{Ba23}}]\label{num-tree}
	For $d\geq 1$, $\beta > 0$, $k\in\mathds{N}$ and ${\bm i}\in(\delta L)\mathds{Z}^d$, $\mathds{E}[|\mathcal{CS}_k({\bm i})|] \leq (4\mu_\beta)^k$, where $\mu_\beta := \mathds{E}[\mathrm{deg}({\bm i})]=\mathds{E}[\mathrm{deg}({\bm 0})]$.
\end{lemma}

Now we present the
\begin{proof}[Proof of upper bound in Theorem~\ref{mr-dim}]
	We begin by fixing $\varepsilon>0, L>0$ and $\delta>0$, and selecting $p\in(0,1/2)$ sufficiently small such that
	\begin{equation}\label{p}
		16\mu_\beta p^{1/(2\cdot 3^d)}\leq 1/2.
	\end{equation}
	Let $\alpha,b$ be chosen from Lemma~\ref{probgood} with $p_0=p$.
Note that $p,\alpha,b$ all depend only  on $d$ and $\beta$.

	For convenience, let $\mathcal{P}_L$ be the collection of all $D$-geodesics contained in $[-L,L]^d$.
We denote $F_{\delta, L,0}$ as the event that
	\begin{equation*}\label{def-F0}
 \sup_{P\in\mathcal{P}_L} N_{\delta,L}(P)>\delta^{-\theta-\varepsilon}
 \end{equation*}
and let $F_{\delta,L,1}$ be the event that all
	$$\Lambda\in\bigcup_{{\bm i}\in(\delta L)\mathds{Z}^d\cap [-L,L]^d}\bigcup_{k>\delta^{-\theta-\varepsilon}}\mathcal{CS}_k({\bm i})$$
	are $(\alpha,b)$-good.
	Then  on the event $F_{\delta, L,0}\cap F_{\delta,L,1}$, we have
	\begin{equation*}
		\mathrm{diam}([-L,L]^d;D)\geq \sup_{P\in \mathcal{P}_L}\mathrm{len}(P)\geq \sup_{P\in \mathcal{P}_L} \left\{N_{\delta, L}(P) (b\alpha \delta L)^\theta/(2\cdot 3^d)\right\}
\geq
\delta^{-\varepsilon} (b\alpha)^\theta  L^\theta/(2\cdot 3^d).
	\end{equation*}
This implies that
	\begin{equation}\label{ineq-step1}
		\begin{aligned}			\mathds{P}[F_{\delta,L,0}]&=\mathds{P}\left[\sup_{P\in\mathcal{P}_L}N_{\delta,L}(P)>\delta^{-\theta-\varepsilon}\right]\\
&\leq \mathds{P}[F_{\delta,L,1}^c]+\mathds{P}[F_{\delta, L,0}\cap F_{\delta,L,1}]\\
			&\leq \mathds{P}[F_{\delta,L,1}^c]+\mathds{P}\left[\mathrm{diam}([-L,L]^d;D)\geq \delta^{-\varepsilon} (b\alpha)^\theta L^\theta/(2\cdot 3^d)\right].
		\end{aligned}
	\end{equation}

Next, we consider an upper bound for the term $\mathds{P}[F_{\delta,L,1}^c]$. The approach is similar to that used in \cite[Lemma~3.11]{DFH23+}. Specifically,
for any $\bm i\in(\delta L)\mathds{Z}^d\cap [-L,L]^d$, let $M_{\bm i}$ be the event that there exists $m\geq \delta^{-\theta-\varepsilon}$ such that $|\mathcal{CS}_m(\bm i)|\geq (8\mu_\beta)^m$.
By Lemma~\ref{num-tree} and Markov's inequality, we have
	\begin{equation*}\label{PMi}
		\begin{aligned}
			\mathds{P}[M_{\bm i}]&\leq \sum_{m>\delta^{-\theta-\varepsilon}}\mathds{P}\left[|\mathcal{CS}_m(\bm i)|\geq (8\mu_\beta)^m\right]\\
			&\leq \sum_{m>\delta^{-\theta-\varepsilon}}(8\mu_\beta)^{-m}\mathds{E}[|\mathcal{CS}_m(\bm i)|] \leq \sum_{m>\delta^{-\theta-\varepsilon}}2^{-m}\leq 2^{-\delta^{-\theta-\varepsilon}+1}.
		\end{aligned}
	\end{equation*}
 Combining this and Lemma~\ref{good-tree}, we obtain
		\begin{align}
			\mathds{P}[F_{\delta,L,1}^c]&\leq \sum_{\bm i\in(\delta L)\mathds{Z}^d\cap [-L,L]^d}\mathds{P}[M_{\bm i}]\nonumber\\
&\quad +\sum_{\bm i\in(\delta L)\mathds{Z}^d\cap [-L,L]^d}\sum_{m>\delta^{-\theta-\varepsilon}}
\mathds{E}\left[\I_{M_{\bm i}^c}\sum_{\Lambda\in\mathcal{CS}_m(\bm i)}\mathds{P}\left[\Lambda\text{ is $(\alpha,b)$-good}\ |G\right]\right]\label{F1C}\\
			&\leq (2\delta^{-1}+1)^d 2^{-\delta^{-\theta-\varepsilon}+1}+ \sum_{\bm i\in(\delta L)\mathds{Z}^d\cap [-L,L]^d}\sum_{m>\delta^{-\theta-\varepsilon}} (8\mu_\beta)^m (2p^{1/(2\cdot 3^d)})^m\nonumber\\
			&\leq (2\delta^{-1}+1)^d 2^{-\delta^{-\theta-\varepsilon}+1}+(2\delta^{-1}+1)^d \sum_{m>\delta^{-\theta-\varepsilon}} 2^{-m}\leq (2\delta^{-1}+1)^d 2^{-\delta^{-\theta-\varepsilon}+2},\nonumber
		\end{align}
where the last inequality uses \eqref{p}.

In addition, according to the tightness of $D$ in \cite[Axiom V in Section 1.3]{DFH23+}, there exists a constant $\widetilde{M}<\infty$ depending only on $d$ and $\beta$ such that
	\begin{equation*}
		\mathds{E}\left[\exp\left\{\frac{\mathrm{diam}([-L,L]^d;D)}{L^\theta}\right\}\right]\leq \widetilde{M}.
	\end{equation*}
	Combining this with Markov's inequality yields that
	\begin{equation}\label{Dtail}
		\mathds{P}\left[\mathrm{diam}([-L,L]^d;D)\geq \delta^{-\varepsilon}(b\alpha)^\theta(2\cdot3^d)^{-1}L^\theta\right]\leq \widetilde{M}\exp\left\{-\delta^{-\varepsilon}(b\alpha)^\theta(2\cdot3^d)^{-1}\right\}.
	\end{equation}

	By applying \eqref{Dtail} and \eqref{F1C} to \eqref{ineq-step1}, we obtain that for each $\varepsilon>0, L>0$ and $\delta>0$
	\begin{equation}\label{ineq-step2}
		\mathds{P}[F_{\delta,L,0}]\leq (2\delta^{-1}+1)^d 2^{-\delta^{-\theta-\varepsilon}+2}+\widetilde{M}\exp\left\{-\delta^{-\varepsilon}(b\alpha)^\theta(2\cdot3^d)^{-1}\right\}.
	\end{equation}

Next, we turn to deriving the upper bound on the Hausdorff dimension from \eqref{ineq-step2}. Indeed, from \eqref{ineq-step2} we see that for any fixed $\varepsilon>0$ and $L>0$
	$$
\sum_{k=1}^\infty \mathds{P}\left[F_{2^{-k}, L,0}\right]<\infty.
$$
Therefore, the Borel-Cantelli Lemma implies that almost surely, there exists a random $k_0\in\mathds{N}$ (depending only on $d, \beta$, $\varepsilon$ and $L$) such that for any $P\in\mathcal{P}_L$ and any $k>k_0$, $N_{2^{-k},L}(P)\leq (2^{k})^{\theta+\varepsilon}$.
Thus, for each $D$-geodesic $P\subset[-L,L]^d$ and each $k>k_0$, we have (recalling that $C_{\Delta}(Y,D)$ is defined in \eqref{hausdorff})
	\begin{equation*}
		C_{\theta+2\varepsilon}(\mathrm{Range}(P))\leq \sum_{\bm k\in\Lambda_{2^{-k},L}(P)} (2^{-k})^{\theta+2\varepsilon}\leq 2^{-k\varepsilon}\to 0\quad \text{as }k\to\infty.
	\end{equation*}
Combining this with the definition of the Hausdorff dimension, we conclude that for any fixed $\varepsilon>0$ and $L>0$, almost surely, $\mathrm{dim}_{\mathrm{H}}(\mathrm{Range}(P))\leq \theta+2\varepsilon$  for all $D$-geodesics $P\subset[-L,L]^d$.
Letting $\varepsilon\to 0$ and then $L\to \infty$ yields the desired upper bound.
\end{proof}

\bigskip

\noindent{\bf Acknowledgement.} \rm
J.\ Ding is partially supported by NSFC Key Program Project No.\ 12231002 and the Xplorer prize. L.-J. Huang is partially supported by National Key R\&D Program of China No.\ 2022YFA1006003, and the National Natural Science Foundation of China No.\ 12471136.

\bibliographystyle{plain}
\bibliography{geodesic-ref}

\begin{thebibliography}{10}

\bibitem{AA97}
P.~Alessandro and S.~Aravind.
\newblock {Randomized distributed edge coloring via an extension of the
  Chernoff-Hoeffding bounds}.
\newblock {\em SIAM J. Comput.}, 26(2):350--368, 1997.

\bibitem{Ba23}
J.~B{\"{a}}umler.
\newblock {Distances in $\frac{1}{|x-y|^{2d}}$ percolation models for all
  dimensions}.
\newblock {\em Comm. Math. Phys.}, 404:1495--1570, 2023.

\bibitem{Ba23a}
J.~B{\"{a}}umler.
\newblock {The polynomial growth of the infinite long-range percolation
  cluster}.
\newblock 2023.
\newblock ArXiv preprint arXiv:2311.14352.

\bibitem{CGS02}
D.~Coppersmith, D.~Gamarnik, and M.~Sviridenko.
\newblock The diameter of a long-range percolation graph.
\newblock {\em Mathematics and Computer Science}, II:147--159, 2002.

\bibitem{DFH23+}
J.~Ding, Z.~Fan, and L.-J. Huang.
\newblock {Uniqueness of the critical long-range percolation metrics}.
\newblock To appear in \textit{Mem. Amer. Math. Soc.}

\bibitem{DFH25}
J.~Ding, Z.~Fan, and L.-J. Huang.
\newblock Polynomial lower bound on the effective resistance for the
  one-dimensional critical long-range percolation.
\newblock {\em Comm. Pure Appl. Math.}, 78(7):1251--1284, 2025.

\bibitem{DG23}
J.~Ding and E.~Gwynne.
\newblock Uniqueness of the critical and supercritical {L}iouville quantum
  gravity metrics.
\newblock {\em Proc. Lond. Math. Soc.}, 126(1):216--333, 2023.

\bibitem{DS13}
J.~Ding and A.~Sly.
\newblock {Distances in critical long range percolation}.
\newblock 2013.
\newblock ArXiv preprint arXiv:1303.3995.

\bibitem{DS20}
J.~Ding and C.K. Smart.
\newblock {Localization near the edge for the Anderson Bernoulli model on the
  two dimensional lattice}.
\newblock {\em Invent. Math.}, 219(2):467--506, 2020.

\bibitem{GGJR14}
S.~Gallo, N.L. Garcia, V.V. Junior, and P.M. Rodr\'{\i}guez.
\newblock Rumor processes on $\mathds{N}$ and discrete renewal processes.
\newblock {\em J. Stat. Phys}, 155:591--602, 2014.

\bibitem{GM21}
E.~Gwynne and J.~Miller.
\newblock Existence and uniqueness of the {L}iouville quantum gravity metric
  for $\gamma\in(0,2)$.
\newblock {\em Invent. Math.}, 223(1):213--333, 2021.

\bibitem{L66}
D.~Lubell.
\newblock {A short proof of Sperner's lemma}.
\newblock {\em J. Comb. Theory}, 1(2):402--402, 1966.

\bibitem{Morters_Peres_2010}
Peter Mörters and Yuval Peres.
\newblock {\em Brownian Motion}.
\newblock Cambridge Series in Statistical and Probabilistic Mathematics.
  Cambridge University Press, 2010.

\bibitem{RV10}
I.~Russell and K.~Valentine.
\newblock Constructive proofs of concentration bounds.
\newblock {\em Approximation, randomization, and combinatorial optimization},
  pages 617--631, 2010.

\bibitem{BK85}
J.~van~den Berg and H.~Kesten.
\newblock Inequalities with applications to percolation and reliability.
\newblock {\em J. Appl. Probab.}, 22(3):556--569, 1985.

\end{thebibliography}

\end{document}